\documentclass[11pt,psfig]{article}

\usepackage{amssymb}
\usepackage{amsmath,amsthm}
\DeclareMathOperator{\vect}{vec}
\usepackage{algorithm}
\usepackage[noend]{algpseudocode}
\usepackage{epsfig,amsthm,amsmath}
\usepackage{verbatim}
\usepackage{algpseudocode}

\algblockdefx[Loop]{Loop}{EndLoop}[1][]{\textbf{Loop} #1}{\textbf{End Loop}}

\usepackage{amsfonts}
\usepackage{graphicx}
\usepackage{graphics}
\usepackage[margin=1.15in]{geometry}
\newtheorem{Theorem}{Theorem}[section]

\newtheorem{lemma}[Theorem]{Lemma}
\newtheorem{example}[Theorem]{Example}
\newtheorem{remark}[Theorem]{Remark}

\newtheorem{theorem}{{\bf Theorem  }}[section]
\newtheorem{prop}{ {\bf Proposition}}[section]

\newtheorem{definition}{{\bf Definition}}[section]

\let\oldproofname=\proofname
\renewcommand{\proofname}{\rm\bf{\oldproofname}}

\DeclareMathOperator{\Mag}{mag}   

\begin{document}

\title{Efficient Approaches for Enclosing the United Solution Set of the Interval Generalized Sylvester Matrix Equations}

\author{Marzieh Dehghani-Madiseh\footnote{
Department of Mathematics, Faculty of Mathematical Sciences and Computer,
Shahid Chamran University of Ahvaz, Ahvaz, Iran, 
e-mail: \texttt{marzieh.dehghani66@gmail.com}
} \and Milan Hlad\'{\i}k\footnote{
Charles University, Faculty  of  Mathematics  and  Physics,
Department of Applied Mathematics, 
Malostransk\'e n\'am.~25, 11800, Prague, Czech Republic, 
e-mail: \texttt{milan.hladik@matfyz.cz}
}}

\maketitle

\begin{abstract}
In this work, we investigate the interval generalized Sylvester matrix equation ${\bf{A}}X{\bf{B}}+{\bf{C}}X{\bf{D}}={\bf{F}}$ and develop some techniques for obtaining outer estimations for the so-called united solution set of this interval system. First, we propose a modified variant of the Krawczyk operator which causes reducing computational complexity to cubic, compared to  Kronecker product form. We then propose an iterative technique for enclosing the solution set. These approaches are based on spectral decompositions of the midpoints of ${\bf{A}}$, ${\bf{B}}$, ${\bf{C}}$ and ${\bf{D}}$ and in both of them we suppose that the midpoints of ${\bf{A}}$ and ${\bf{C}}$ are simultaneously diagonalizable as well as for the midpoints of the matrices ${\bf{B}}$ and ${\bf{D}}$. Some numerical experiments are given to illustrate the performance of the proposed methods.
\end{abstract}

\begin{quote}\textbf{Keywords:} 
Interval arithmetic; Generalized Sylvester matrix equation; Krawczyk operator; Preconditioning.
\end{quote}

\section{Introduction}\label{section 1}
Consider the implicit differential equation
\begin{equation}\label{eq.1.1}
  g(\dot{x},x)=0.
\end{equation}
As said in \cite{Ept1980}, for obtaining a numerical solution of (\ref{eq.1.1}) using the block multistep methods suppose one has available quantities $(\dot{x}_{j,n-p},x_{j,n-p})$ that approximate $\dot{x}(t)$ and $x(t)$ at past times $t_{j,n-p}=t_{n-p-1}+hc_{j}$, where $h$ is the time step, $j=1,\ldots,v$ and $p=1,\ldots,k$ (usually the numbers $c_{j}$ satisfy $0\leq c_{j}\leq 1$). To advance the method from time $t_{n-1}$ to time $t_{n}$, the quantities $x_{j,n},\dot{x}_{j,n}$ at stage $l$ of the iterations must be satisfied in the following conditions
\begin{equation}\label{eq.1.2}
  g(\dot{x}^{(l)}_{j,n},x^{(l)}_{j,n})+[\partial g/\partial \dot{x}]\delta\dot{x}_{j,n}^{(l)}+[\partial g/\partial x]\delta x_{j,n}^{(l)}=0, \hspace{.5 cm} j=1,\ldots,v,
\end{equation}
\begin{equation}\label{eq.1.3}
  r_{i}^{(l)}+\sum_{j=1}^{v}(\alpha_{ij}^{(0)} \delta x_{j,n}^{(l)}-h\beta_{ij}^{(0)}\delta\dot{x}_{j,n}^{(l)})=0,
\end{equation}
for specified parameters $\alpha_{ij},\beta_{ij},r_{i},\delta$, see \cite{Ept1980}. But solution of the pairs (\ref{eq.1.2}) and (\ref{eq.1.3}) may be obtained by solving a pair of generalized Sylvester matrix equations as follows
\begin{displaymath}
 \left\{ \begin{array}{ll}
A\delta X^{(l)}D -B\delta X^{(l)}C=-AR^{(l)}+G^{(l)}C, \\
A\delta \dot{X}^{(l)}D -B\delta \dot{X}^{(l)}C=(1/h)BR^{(l)}-G^{(l)}D.
\end{array} \right.
\end{displaymath}

This motivates us to consider the generalized Sylvester matrix equation
\begin{equation}\label{eq.1.4}
  AXB+CXD=F,
\end{equation}
where $A,C$ are square known matrices of order $m$, $B,D$ are square known matrices of order $n$, and the right-hand side matrix $F$ and the unknown matrix $X$ are $m$-by-$n$ matrices. Each of the introduced matrices can be real or complex. This equation has nice applications in various branches of science. Equations in the form (\ref{eq.1.4}) appear in the study of perturbations of the generalized eigenvalue problem \cite{Chu1987}, in MINQUE theory of estimating covariance components in a covariance model \cite{Rao1972}, in stability problems for descriptor systems \cite{Ben2012} and in the numerical solution of implicit ordinary differential equations \cite{Ept1980,HerGas1989}. Also equation (\ref{eq.1.4}) includes two important linear problems in the space of matrices, namely Lyapunov and Sylvester matrix equations that have vital roles in many areas of mathematics and engineering such as in control theory, stability and robust stability, image processing, model reduction and many other applications, see \cite{FroHas2012} and the references therein. Some methods for solving the matrix equation (\ref{eq.1.4}) can be found in \cite{HerGas1989,Chu1987b,DehHaj2010}.

Though the matrix equations of the form (\ref{eq.1.4}) are studied in the literature, less or even no attention has been paid to the form of uncertainties that may occur in the elements of $A,B,C,D$ and $F$. In fact, in practice the elements of input data are obtained from the experience and so due to the measurement errors they will be accompanied by uncertainty. It is natural to describe these uncertainties by intervals and hence we will have the interval generalized Sylvester matrix equation
\begin{equation}\label{eq.1.5}
  {\bf{A}}X{\bf{B}}+{\bf{C}}X{\bf{D}}={\bf{F}},
\end{equation}
where ${\bf{A}}$, ${\bf{B}}$, ${\bf{C}}$, ${\bf{D}}$ and ${\bf{F}}$ are interval matrices (boldface letters stand for interval quantities). Interval computations can be used in various areas such as set inversion, motion planning, robotics, traffic control, electronic engineering, economics \cite{Kea1996b,KeaKre1996}.

In general, computing the exact solution of an interval linear system is NP-hard \cite{RohKre1995}, and so many researchers are interested in finding some approximations for the solution set. Up to now, only a few techniques for approximating the solution set of interval matrix equations have been proposed. Different techniques for enclosing the united solution set of the interval Sylvester matrix equation ${\bf{A}}X+X{\bf{B}}={\bf{C}}$ have been examined by Seif et al. \cite{SeiHus1994}. Shashikhin \cite{Shas2002,Shas2002b} used the correspondence between the interval Sylvester matrix equation  ${\bf{A}}X+X{\bf{B}}={\bf{C}}$ and an interval linear system of the following form to find an interval enclosure for the united solution set
\begin{equation*}
  ((I_{n}\otimes{\bf{A}})+({\bf{B}}^{\top}\otimes I_{m}))x={\bf{c}}, \hspace{.5 cm} x={\vect}(X),  \hspace{.5 cm} {\bf{c}}={\vect}({\bf{C}}),
\end{equation*}
in which $\otimes$ denotes the Kronecker product and ${\vect}({\bf{F}})$ is an $mn$-dimensional vector obtained by stacking the columns of matrix ${\bf{F}}$, i.e., 
$${\vect}({\bf{F}})=({\bf{F}}_{11},\ldots,{\bf{F}}_{m1},\ldots,{\bf{F}}_{1n},\ldots, {\bf{F}}_{mn})^{\top}.$$
Hashemi and Dehghan \cite{HasDeh2011} used an interval Gaussian elimination method for enclosing the united solution set of ${\bf{A}}X={\bf{B}}$ and also in \cite{HasDeh2012} proposed a modification of Krawczyk operator with a significant reduction in computational complexity of obtaining an outer estimation of the united solution set to the interval Lyapunov matrix equation ${\bf{A}}X+X{\bf{A}}^{\top}={\bf{F}}$. Dehghani-Madiseh and Dehghan in \cite{DehDeh2014} developed some algebraic and numerical techniques for obtaining inner and outer estimations for the generalized solution sets of the following interval equation
\begin{equation*}
\sum_{i=1}^{p}{\bf{A}}_{i}X_{i}+\sum_{j=1}^{q}{Y_{j}{\bf{B}}_{j}}={\bf{C}}.
\end{equation*}
Rivaz et al. \cite{RivMog2014} considered the system of interval matrix equations
\begin{displaymath}
 \left\{ \begin{array}{ll}
{\bf{A}}_{11}X+Y{\bf{A}}_{12}={\bf{C}}_{1}, \\
{\bf{A}}_{21}X+Y{\bf{A}}_{22}={\bf{C}}_{2},
\end{array} \right.
\end{displaymath}
and presented direct and iterative approaches for enclosing its united solution set.

The interval generalized Sylvester matrix equation (\ref{eq.1.5}) can be transformed to the interval linear system
\begin{equation}\label{eq.1.6}
  {\bf{Q}}x={\bf{f}},
\end{equation}
in which ${\bf{Q}}={\bf{B}}^{\top}\otimes {\bf{A}}+{\bf{D}}^{\top}\otimes{\bf{C}}$, $x=\textrm{vec}(X)$ and ${\bf{f}}=\textrm{vec}({\bf{F}})$. It is to be noted that the interval matrix ${\bf{Q}}$ has a special structure. In fact, its elements have non-linear dependencies and the interval linear system (\ref{eq.1.6}) is a parametric interval linear system. For more details about parametric systems see, e.g., \cite{Hla2012d,PopKra2007}. But the common approach when considering the transformed system of equations (\ref{eq.1.6}) is to treat with it as a non-parametric interval linear system, i.e., the elements of the coefficient matrix ${\bf{Q}}$ and the right-hand side interval vector ${\bf{f}}$ are supposed to vary independently. Based on this choice, the most commonly used approach for computing enclosure for the solution set of an interval matrix equation of type (\ref{eq.1.5}) is to firstly transform it into an interval linear system of the form (\ref{eq.1.6}) and then using a technique for enclosing the solution set of that interval linear system. This is exactly the idea that was proposed by Rohn \cite{versoft} in the \verb"VERMATREQN.m" code of the \verb"VERSOFT" software. But this approach has a computational complexity of $O(m^{3}n^{3})$ which is very high even for small sizes of interval system (\ref{eq.1.5}). In this work, we want to present some approaches that reduce the cost to $O(m^{3}+n^{3})$.

\subsubsection*{Notation.}
In this work, boldface letters denote interval quantities and ordinary letters stand for real quantities. Notations $\mathbb{R}$ and $\mathbb{C}$,  respectively stand for the field of real and complex numbers. We use $\mathbb{K}$ to denote either of the fields $\mathbb{R}$ or $\mathbb{C}$. In the case of $\mathbb{K}=\mathbb{R}$ ($\mathbb{K}=\mathbb{C}$), $\mathbb{IK}$ denotes the space of real (complex) intervals. Further, $\mathbb{IK}^{n}$ and $\mathbb{IK}^{m\times n}$ stand for the set of all $n$-dimensional interval vectors and the set of all $m$-by-$n$ interval matrices over field $\mathbb{K}$, respectively. For an interval quantity ${\bf{x}}\in\mathbb{IK}$, we use $x^{c}$ or mid(${\bf{x}}$) for denoting its midpoint, and $x^{\Delta}$ or rad(${\bf{x}}$)  stand for its radius. The magnitude of ${\bf{x}}$ is defined as $\Mag({\bf{x}})\equiv |{\bf{x}}|=\max\{|x|:x\in{\bf{x}}\}=|x^{c}|+x^{\Delta}$. For a square matrix $A=(A_{ij})\in\mathbb{K}^{m\times m}$, diag($A$)=$(A_{11},\ldots,A_{mm})^{\top}$ denotes its diagonal and for a vector $a\in\mathbb{K}^{m}$, Diag($a$)=Diag$(a_{1},\ldots,a_{m})$ is the diagonal matrix with diagonal entries $a_{i}$, $i=1,\ldots,m$. The diagonal part of matrix $A$ is Diag(diag($A$)).

Two most frequently used representations for intervals over $\mathbb{K}\in\{\mathbb{R},\mathbb{C}\}$ are as follow:

{\textbf{(i)}} The infimum-supremum representation
\begin{equation}\label{eq.1.7}
  [\underline{x},\overline{x}]=\{x\in\mathbb{K}:\underline{x}\leq x\leq \overline{x}\}, \hspace{.5 cm} {\textrm{for some}} \hspace{.2 cm} \underline{x},\overline{x}\in\mathbb{K}, \hspace{.3 cm} \underline{x}\leq\overline{x},
\end{equation}
where $\leq$ is the partial ordering $x\leq y \Leftrightarrow {\textrm{Re}}(x)\leq {\textrm{Re}}(y)$ \& ${\textrm{Im}}(x)\leq {\textrm{Im}}(y) $, for $x,y\in{\mathbb{C}}$, in which  ${\textrm{Re}}(x)$ and ${\textrm{Im}}(x)$, respectively stand for the real part and the imaginary part of $x$.

{\textbf{(ii)}} The midpoint-radius representation
\begin{equation}\label{eq.1.8}
\bigl<  x^{c},x^{\Delta} \bigr >=\{x\in\mathbb{K}: |x-x^{c}|\leq x^{\Delta} \}   \hspace{.5 cm} {\textrm{for some}} \hspace{.2 cm} x^{c}\in\mathbb{K}, \hspace{.3 cm} 0\leq x^{\Delta}\in\mathbb{R}.
\end{equation}
The two representations are identical for real intervals, whereas for complex intervals the first representation are rectangles, the second one represents disks in the complex plain, see \cite{Rum1999b}. The midpoint-radius representation $\bigl<  x^{c},x^{\Delta} \bigr >$ for interval number ${\bf{x}}$, also is denoted by midrad($x^{c},x^{\Delta}$).

There are different definitions for the basic arithmetic operations $\circ \in\{+,-,*,/\}$ over $\mathbb{IK}$ dependent on referring by $\mathbb{IK}$ to intervals in infimum-supremum representation (\ref{eq.1.7}) or midpoint-radius representation (\ref{eq.1.8}). But if we consider the last case (complex circular arithmetic) then interval operations satisfy the following fundamental property of inclusion isotonicity
\begin{equation*}
\{x\circ y: x\in{\bf{x}}, y\in{\bf{y}} \} \subseteq {\bf{x}}\circ {\bf{y}}.
\end{equation*}
As said in \cite{FroHas2012}, for reasons of computational efficiency, the interval package Intlab \cite{Rum1999} uses the restriction of complex circular arithmetic to the real axis as its default arithmetic for real intervals. This results in a different multiplication and division as in standard real interval arithmetic. For the purposes of this work, we do not depend on the particular interval arithmetic in use. All we need is the enclosure property to hold, and this is true for standard real arithmetic, complex circular arithmetic as well as Intlab's default real arithmetic.

The paper is organized as follows: In Section~\ref{section 2}, we have our main results. First, we define the concept of united solution set for interval system (\ref{eq.1.5}) and then give a sufficient condition under which this solution set is bounded. We use a preconditioning technique and then present a modification of the Krawczyk operator applied on the preconditioned system. Also a block diagonalization approach for the ill-conditioned cases will be proposed. We then present an iterative technique for enclosing the solution set of the interval system (\ref{eq.1.5}). Section~\ref{section 3} contains some numerical tests and comparisons. Finally, the paper is completed by a short conclusion in Section~\ref{section 4}.

\section{Some efficient approaches for enclosing the united solution set}\label{section 2}

\subsection{The united solution set}
Equation (\ref{eq.1.5}) is interpreted as a collection of all generalized Sylvester matrix equations $AXB+CXD=F$, where the coefficients vary in intervals, i.e., $A\in{\bf{A}}$, $B\in{\bf{B}}$, $C\in{\bf{C}}$, $D\in{\bf{D}}$ and $F\in{\bf{F}}$. For an interval linear system, different types of solution sets can be considered. The concept of generalized solution sets for a system of interval linear equations was introduced for the first time by Shary \cite{Sha1999}. But the united solution set is the widest in the collection of all generalized solution sets to an interval linear system. Also the united solution set has numerous applications in verified numerical computations based on the interval analysis. So it is natural that researchers pay special attention to this type of solution set. We define the united solution set to the interval generalized Sylvester matrix equation (\ref{eq.1.5}) as follows.
\begin{definition}\label{def.2.1}
The united solution set to the interval generalized Sylvester matrix equation (\ref{eq.1.5}) is the set
\begin{equation}\label{eq.2.1}
\Xi=\{X\in\mathbb{K}^{m\times n}: (\exists A\in{\bf{A}})(\exists B\in{\bf{B}})(\exists C\in{\bf{C}})(\exists D\in{\bf{D}})(\exists F\in{\bf{F}})(AXB+CXD=F)\}.
\end{equation}
\end{definition}

It is hard to characterize $\Xi$ by simple means (as is common for parametric interval system in general; cf. \cite{Hla2012d}). However, we have the following necessary conditions.

\begin{prop}\label{proposition 2.1}
If $X\in\Xi$, then
\begin{equation}\label{aeq 1}
|A^{c}XB^{c}+C^{c}XD^{c}-F^{c}|\leq \Mag({\bf{A}})|X|B^{\Delta}+  A^{\Delta}|X B^{c}|  +\Mag({\bf{C}})|X|D^{\Delta}+  C^{\Delta}|X D^{c}|+ F^{\Delta}.
\end{equation}
\end{prop}

\begin{proof}
Let $X\in\Xi$, then there exist $A\in{\bf{A}}$, $B\in{\bf{B}}$, $C\in{\bf{C}}$, $D\in{\bf{D}}$ and $F\in{\bf{F}}$ such that $AXB+CXD=F$, which implies
{\setlength\arraycolsep{2pt}
\begin{eqnarray}\label{equ 1}
 |A^{c}XB^{c}+C^{c}XD^{c}-F^{c}| &&=|(AXB-A^{c}XB^{c})+(CXD-C^{c}XD^{c})-(F-F^{c})|
\nonumber\\
&& \leq |AXB-A^{c}XB^{c}|+|CXD-C^{c}XD^{c}|+|F-F^{c}|.
\end{eqnarray}}
On the other hand, we have
{\setlength\arraycolsep{2pt}
\begin{eqnarray}\label{equ 2}
 |AXB-A^{c}XB^{c}| && =|AXB-AXB^{c}+AXB^{c}-A^{c}XB^{c}|
 \nonumber\\
&& \leq |AXB-AXB^{c}|+|AXB^{c}-A^{c}XB^{c}|
\nonumber\\
&& \leq |AX||B-B^{c}|+|A-A^{c}||XB^{c}| \leq |AX|B^{\Delta}+A^{\Delta}|XB^{c}|
\nonumber\\
&& \leq |A||X|B^{\Delta}+A^{\Delta}|XB^{c}|\leq \Mag({\bf{A}})|X|B^{\Delta}+  A^{\Delta}|X B^{c}|,
\end{eqnarray}}
and similarly
\begin{equation}\label{equ 3}
|CXD-C^{c}XD^{c}|\leq \Mag({\bf{C}})|X|D^{\Delta}+  C^{\Delta}|X D^{c}|.
\end{equation}
Using (\ref{equ 1}), (\ref{equ 2}), (\ref{equ 3}) and $|F-F^{c}|\leq F^{\Delta}$, we obtain (\ref{aeq 1}).
\end{proof}

\begin{remark}\label{remark 1}
Due to the symmetry in the construction of inequality (\ref{aeq 1}), we conclude that if $X$ belongs to the solution set of the interval generalized Sylvester matrix equation (\ref{eq.1.5}) then $X$ will be satisfied also in each of the following inequalities
\begin{align*}
 |A^{c}XB^{c}+C^{c}XD^{c}-F^{c}|&\leq |A^{c}X|B^{\Delta}+A^{\Delta}|X| \Mag({\bf{B}}) +\Mag({\bf{C}})|X|D^{\Delta}+  C^{\Delta}|X D^{c}| +F^{\Delta},
 \\
 |A^{c}XB^{c}+C^{c}XD^{c}-F^{c}|&\leq \Mag({\bf{A}})|X|B^{\Delta}+  A^{\Delta}|X B^{c}| + |C^{c}X|D^{\Delta}+C^{\Delta}|X| \Mag({\bf{D}})+F^{\Delta} ,
\\
 |A^{c}XB^{c}+C^{c}XD^{c}-F^{c}|&\leq |A^{c}X|B^{\Delta}+A^{\Delta}|X| \Mag({\bf{B}}) +|C^{c}X|D^{\Delta}+C^{\Delta}|X| \Mag({\bf{D}})+F^{\Delta} ,
\end{align*}
\end{remark}
In this work, we want to obtain an interval matrix that encloses the solution set $\Xi$ defined by (\ref{eq.2.1}). It is an outer estimation problem. But this enclosure is achievable if $\Xi$ is a bounded set. In Theorem \ref{theorem 2.1} below, we give a sufficient condition under which the solution set (\ref{eq.2.1}) is bounded.
\begin{theorem}\label{theorem 2.1}
For any $m$-by-$n$ interval matrix ${\bf{F}}$, the solution set $\Xi$ to the interval generalized Sylvester matrix equation (\ref{eq.1.5}) is bounded if the inequality
\begin{equation}\label{eq.2.4}
|A^{c}XB^{c}+C^{c}XD^{c}|\leq \Mag({\bf{A}})|X|B^{\Delta}+  A^{\Delta}|X B^{c}|  +\Mag({\bf{C}})|X|D^{\Delta}+  C^{\Delta}|X D^{c}|
\end{equation}
has only the trivial solution $X=0\in\mathbb{R}^{m\times n}$.
\end{theorem}

\begin{proof}
Using Proposition \ref{proposition 2.1}, we know that if $X$ belongs to the solution set of the interval generalized Sylvester matrix equation ${\bf{A}}X{\bf{B}}+{\bf{C}}X{\bf{D}}=0$, then $X$ satisfies in the inequality (\ref{eq.2.4}). According to the assumption of the theorem, inequality (\ref{eq.2.4}) has only the trivial solution $X=0\in\mathbb{R}^{m\times n}$. So the solution set of the interval system ${\bf{A}}X{\bf{B}}+{\bf{C}}X{\bf{D}}=0$ is $\{0\}$ and this means that for all $A\in{\bf{A}}$, $B\in{\bf{B}}$, $C\in{\bf{C}}$ and $D\in{\bf{D}}$, equation $AXB+CXD=0$ has only the trivial solution $X=0\in\mathbb{R}^{m\times n}$. Thus for all $A\in{\bf{A}}$, $B\in{\bf{B}}$, $C\in{\bf{C}}$ and $D\in{\bf{D}}$, its equivalent form $(B^{\top}\otimes A+D^{\top}\otimes C){\textrm{vec}}(X)=0$ has only the trivial solution ${\textrm{vec}}(X)=0\in\mathbb{R}^{mn}$ that yields for all $A\in{\bf{A}}$, $B\in{\bf{B}}$, $C\in{\bf{C}}$ and $D\in{\bf{D}}$, the matrix $B^{\top}\otimes A+D^{\top}\otimes C$ is nonsingular. Hence for any $m$-by-$n$ interval matrix ${\bf{F}}$ the set
\begin{equation*}
\{{\textrm{vec}}(X): (B^{\top}\otimes A+D^{\top}\otimes C){\textrm{vec}}(X)={\textrm{vec}}(F), A\in{\bf{A}}, B\in{\bf{B}},C\in{\bf{C}},D\in{\bf{D}},F\in{\bf{F}} \}
\end{equation*}
is bounded. Bringing the above set back to its equivalent form, i.e., the set
\begin{equation*}
\{X:AXB+CXD=F , A\in{\bf{A}}, B\in{\bf{B}},C\in{\bf{C}},D\in{\bf{D}},F\in{\bf{F}} \},
\end{equation*}
implies that the solution set of the interval generalized Sylvester matrix equation ${\bf{A}}X{\bf{B}}+{\bf{C}}X{\bf{D}}={\bf{F}}$ is bounded for any desirable interval matrix ${\bf{F}}\in\mathbb{IK}^{m\times n}$.
\end{proof}

\begin{remark}\label{remark 2}
Due to Remark \ref{remark 1}, we can say that if any of the inequalities
{\setlength\arraycolsep{2pt}
\begin{eqnarray*}
&& |A^{c}XB^{c}+C^{c}XD^{c}|\leq |A^{c}X|B^{\Delta}+A^{\Delta}|X|\Mag({\bf{B}}) +\Mag({\bf{C}})|X|D^{\Delta}+  C^{\Delta}|X D^{c}|,
 \nonumber\\
&& |A^{c}XB^{c}+C^{c}XD^{c}|\leq \Mag({\bf{A}})|X|B^{\Delta}+  A^{\Delta}|X B^{c}| + |C^{c}X|D^{\Delta}+C^{\Delta}|X| \Mag({\bf{D}}),
\nonumber\\
&& |A^{c}XB^{c}+C^{c}XD^{c}|\leq |A^{c}X|B^{\Delta}+A^{\Delta}|X| \Mag({\bf{B}}) +|C^{c}X|D^{\Delta}+C^{\Delta}|X|\Mag({\bf{D}}) ,
\end{eqnarray*}}
has only the trivial solution $X=0\in\mathbb{R}^{m\times n}$ then for any $m$-by-$n$ interval matrix ${\bf{F}}$, the solution set $\Xi$ to the interval generalized Sylvester matrix equation (\ref{eq.1.5}) will be bounded.
\end{remark}

Now, we describe the midpoint and radius of the coefficient matrix ${\bf{Q}}$ in equation (\ref{eq.1.6}) using the interval operations over representation (\ref{eq.1.8}). This description will be used for theoretical considerations in the sequel. Interval matrices ${\bf{A}}$, ${\bf{B}}$, ${\bf{C}}$ and ${\bf{D}}$ based on the midpoint-radius representation (\ref{eq.1.8}) have the following form
\begin{equation}\label{eq.2.8}
{\bf{A}}=\big< A^{c},A^{\Delta} \big>, \hspace{.4 cm} {\bf{B}}=\big< B^{c},B^{\Delta} \big>, \hspace{.4 cm} {\bf{C}}=\big< C^{c},C^{\Delta} \big>, \hspace{.4 cm} {\bf{D}}=\big< D^{c},D^{\Delta} \big>.
\end{equation}
Now, for determining ${\bf{Q}}={\bf{B}}^{\top}\otimes{\bf{A}}+{\bf{D}}^{\top}\otimes{\bf{C}}$, first we determine ${\bf{B}}^{\top}\otimes{\bf{A}}$. In fact ${\bf{B}}^{\top}\otimes{\bf{A}}$ is an $mn$-by-$mn$ block matrix whose its $(s,t)$-th block is $m$-by-$m$ matrix ${\bf{B}}^{\top}_{st}{\bf{A}}$ for $s=1,\ldots,n$ and $t=1,\ldots,n$. We call this block ${\bf{T}}^{st}$. For $i=1,\ldots,m$ and $j=1,\ldots,m$, ${\bf{T}}^{st}_{ij}$ is
\begin{equation*}
{\bf{T}}^{st}_{ij}={\bf{B}}^{\top}_{st}{\bf{A}}_{ij}=\big< B_{st}^{\top c},B_{st}^{\top \Delta} \big> \big< A_{ij}^{ c},A_{ij}^{\Delta} \big> =\big< B_{st}^{\top c} A_{ij}^{ c} ,|B_{st}^{\top c}|A_{ij}^{\Delta}+ B_{st}^{\top \Delta}|A_{ij}^{c}|+ B_{st}^{\top \Delta} A_{ij}^{\Delta}  \big>.
\end{equation*}
So it is easy to see that
\begin{equation*}
{\bf{B}}^{\top}\otimes{\bf{A}}=\big< B^{\top c} \otimes A^{ c} ,|B^{\top c}|\otimes A^{\Delta}+ B^{\top \Delta}\otimes|A^{c}|+ B^{\top \Delta}\otimes A^{\Delta}  \big>.
\end{equation*}
In a similar way we have
\begin{equation*}
{\bf{D}}^{\top}\otimes{\bf{C}}=\big< D^{\top c}\otimes C^{ c} ,|D^{\top c}|\otimes C^{\Delta}+ D^{\top \Delta}\otimes|C^{c}|+ D^{\top \Delta}\otimes C^{\Delta}  \big>,
\end{equation*}
and so
\begin{equation}\label{eq.2.9}
{\bf{Q}}=\big< B^{\top c}\otimes A^{ c}+ D^{\top c} \otimes C^{ c} ,|B^{\top c}|\otimes A^{\Delta} + B^{\top \Delta}\otimes\Mag({\bf{A}})+   |D^{\top c}|\otimes C^{\Delta}+ D^{\top \Delta}\otimes\Mag({\bf{C}}) \big>.
\end{equation}

\begin{lemma}\label{lemma 2.1}
For any three point matrices $A$, $B$ and $C$ of compatible sizes, we have\\

(i) ${\vect}(ABC)=(C^{\top}\otimes A){\vect}(B)$,\\

(ii) $(Diag({\vect}(A)))^{-1}{\vect}(B)={\vect}(B./A)$, where $./$ denotes the Hadamard division.
\end{lemma}

Part (i) is from \cite{HorJoh1991} and part (ii) is taken from \cite{FroHas2009}. Note that part (ii) holds also for interval matrices but due to sub-distributivity low in interval arithmetic part (i) generally does not hold for interval matrices. However, due to the enclosure property of interval arithmetic we have the following lemma. This lemma also has been expressed in \cite{FroHas2012}.

\begin{lemma}\label{lemma 2.2}
Let ${\bf{A}}$, ${\bf{B}}$ and ${\bf{C}}$ are interval matrices of compatible sizes. Then\\
\begin{equation*}
\{(C^{\top}\otimes A){\vect}(B): A\in{\bf{A}}, B\in{\bf{B}}, C\in{\bf{C}} \}\subseteq  \left\{ \begin{array}{ll}
 {\vect}(({\bf{A}}{\bf{B}}){\bf{C}}),\\
{\vect}({\bf{A}}({\bf{B}}{\bf{C}})).
\end{array} \right.
\end{equation*}
\end{lemma}

In general, the solution set (\ref{eq.2.1}) has a very complicated structure. 
If $\bf{A}$, $\bf{B}$, $\bf{C}$ and $\bf{D}$ contain only real intervals and if
$X$ belongs to the solution set (\ref{eq.2.1}) of the interval generalized Sylvester matrix equation (\ref{eq.1.5}), then using Proposition \ref{proposition 2.1}, we can rewrite (\ref{aeq 1}) as
\begin{displaymath}
\left\{ \begin{array}{ll}
A^{c}XB^{c}+C^{c}XD^{c} - \Mag({\bf{A}})|X|B^{\Delta}-  A^{\Delta}|X|| B^{c}|  -\Mag({\bf{C}})|X|D^{\Delta}-  C^{\Delta}|X|| D^{c}| \leq  \overline{F},\vspace{.3 cm}\\
A^{c}XB^{c}+C^{c}XD^{c} +\Mag({\bf{A}})|X|B^{\Delta}+  A^{\Delta}|X|| B^{c}|  +\Mag({\bf{C}})|X|D^{\Delta}+  C^{\Delta}|X ||D^{c}| \geq \underline{F}.
\end{array} \right.
\end{displaymath}
Now, if we define the sign matrix $M$ corresponding to the solution matrix $X$ such that $M_{ij}={\textrm{sgn}}(X_{ij})$ then it is obvious that $|X|=M\circ X$, where $\circ$ stands for the Hadamard product. So the following linear programming problems can be solved to minimize each component $X_{ij}$
\begin{equation*}
\begin{array}{ll@{}ll}
\text{minimize} \hspace{.2 cm} X_{ij} &\\
\text{subject to}& \\
                \left\{ \begin{array}{ll}
A^{c}XB^{c}+C^{c}XD^{c} -\mathcal{T}(X) \leq \overline{F},\vspace{.1 cm}\\
A^{c}XB^{c}+C^{c}XD^{c} + \mathcal{T}(X) \geq \underline{F},
\end{array} \right.
\end{array}
\end{equation*}
for all possible sign matrices $M$ and therein $\mathcal{T}(X)=\Mag({\bf{A}})(M\circ X)B^{\Delta}+  A^{\Delta}(M\circ X)| B^{c}|  +\Mag({\bf{C}})(M\circ X)D^{\Delta}+  C^{\Delta}(M\circ X)|D^{c}|$. The similar linear programming problems must be solved to maximize each component $X_{ij}$. This approach encloses the precise solution set of the interval generalized Sylvester matrix equation (\ref{eq.1.5}). However, there are $2mn\times 2^{mn}$ linear programming problems to be solved which makes the problem very troublesome even for small values of $m$ and $n$. This motivates us to propose another methods for enclosing the solution set (\ref{eq.2.1}) with lower computational costs, but of course on account of possible loss of tightness.

\subsection{A modified Krawczyk method for enclosing the solution set}\label{subsection 2.2}
For an interval linear system ${\bf{A}}x={\bf{b}}$ with  given real vector $\tilde{x}$ and real matrix $R$, the standard Krawczyk operator ${\bf{k}}(\tilde{x},{\bf{x}})$ is defined as
\begin{equation*}
{\bf{k}}(\tilde{x},{\bf{x}})=\tilde{x}-R({\bf{A}}\tilde{x}-{\bf{b}})+(I-R{\bf{A}})({\bf{x}}-\tilde{x}), \hspace{.4 cm} \tilde{x}\in{\bf{x}},
\end{equation*}
in which $I$ stands for the identity matrix. The following theorem gives the important properties of this operator.

\begin{theorem}\cite{JanRum1991,Rum1994}\label{theorem 2.2}
Let ${\bf{A}}\in{\mathbb{IR}}^{n\times n}$, ${\bf{b}}\in{\mathbb{IR}}^{n}$ be given  and let for some $R\in{\mathbb{R}}^{n\times n}$, $\tilde{x}\in{\mathbb{R}}^{n}$ and ${\bf{x}}\in{\mathbb{IR}}^{n}$
\begin{equation*}
{\bf{k}}(\tilde{x},{\bf{x}}+\tilde{x})
\subseteq {\textrm{int}}({\bf{x}}+\tilde{x}).
\end{equation*}
Then, $R$ and every matrix $A\in{\bf{A}}$ is nonsingular and for all $A\in{\bf{A}}$, $b\in{\bf{b}}$ the corresponding linear system $Ax=b$ is uniquely solvable with solution $\hat{x}$ satisfying $\hat{x}\in\tilde{x}+{\bf{x}}$. Therefore, $\Xi({\bf{A}},{\bf{b}})\subseteq\tilde{x}+{\bf{x}}$.
\end{theorem}

In the above theorem $\Xi({\bf{A}},{\bf{b}})$ stands for the solution set of the interval linear system ${\bf{A}}x={\bf{b}}$. The Krawczyk operator when applied to the transformed system ${\bf{Q}}x={\bf{f}}$ in (\ref{eq.1.6}) will be very costly to evaluated. In this case, it is
\begin{equation*}
{\bf{k}}(\tilde{x},{\bf{x}})=\tilde{x}-R[({\bf{B}}^{\top}\otimes{\bf{A}}+{\bf{D}}^{\top}\otimes{\bf{C}})\tilde{x}-{\bf{f}}]+ [I_{mn}-R({\bf{B}}^{\top}\otimes{\bf{A}}+ {\bf{D}}^{\top}\otimes{\bf{C}})]({\bf{x}}-\tilde{x}), \hspace{.4 cm} \tilde{x}\in{\bf{x}},
\end{equation*}
where $R\in{\mathbb{K}}^{mn\times mn}$ usually is a computed (approximate) inverse of ${\textrm{mid}}({\bf{B}}^{\top}\otimes{\bf{A}}+{\bf{D}}^{\top}\otimes{\bf{C}})$. Since ${\bf{B}}^{\top}\otimes{\bf{A}}+{\bf{D}}^{\top}\otimes{\bf{C}}$ is an $mn$-by-$mn$ matrix, it is obvious that computing such an approximate inverse is very costly. This motivates us to develop a modification of Krawczyk type method reducing the computational costs considerably. Our approach consists of applying a Krawczyk type method on a preconditioned system obtained from digonalizations of the midpoint of interval matrices ${\bf{A}}$, ${\bf{B}}$, ${\bf{C}}$ and ${\bf{D}}$. For obtaining outer estimations for the solution set of the interval generalized Sylvester matrix equation (\ref{eq.1.5}), we use the following theorem that has been proved by Frommer and Hashemi \cite{FroHas2009} and describes the substance of all Krawczyk type verification methods for a general function $f\colon D\subset{\mathbb{C}}^{N}\rightarrow {\mathbb{C}}^{N}$. A mapping
\begin{equation*}
A \colon D\times D \rightarrow {\mathbb{C}}^{N\times N},
\end{equation*}
is called a slope for $f$ if
\begin{equation*}
f(y)-f(x)=A(y,x)(y-x)  \hspace{.5 cm} {\textrm{for all $x,y\in D$}}.
\end{equation*}

\begin{theorem}\cite{FroHas2009}\label{theorem 2.3}
Assume that $f\colon D\subset{\mathbb{C}}^{N}\rightarrow {\mathbb{C}}^{N}$ is continuous in $D$. Let $\tilde{x}\in D$ and ${\bf{x}}\in{\mathbb{IC}}^{N}$ be such that $\tilde{x}+{\bf{x}}\in D$. Moreover, assume that $\mathcal{A}\subset{\mathbb{C}}^{N\times N}$ is a set of matrices containing all slopes $A(\tilde{x},y)$ for $y\in\tilde{x}+{\bf{x}}=:{\bf{z}}$. Finally, let $R\in{\mathbb{C}}^{N\times N}$. Denote $\mathcal{K}_{f}(\tilde{x},R,{\bf{x}},\mathcal{A})$ the set
\begin{equation}\label{eq.2.12}
\mathcal{K}_{f}(\tilde{x},R,{\bf{x}},\mathcal{A}):=\{-Rf(\tilde{x})+(I-RA)x: A\in\mathcal{A}, x\in{\bf{x}}\}.
\end{equation}
Then if
\begin{equation}\label{eq.2.13}
\mathcal{K}_{f}(\tilde{x},R,{\bf{x}},\mathcal{A})\subseteq {\textrm{int}}({\bf{x}}),
\end{equation}
the function $f$ has a zero $x^{*}$ in $\tilde{x}+\mathcal{K}_{f}(\tilde{x},R,{\bf{x}},\mathcal{A})\subseteq{\bf{z}}$. Moreover, if $\mathcal{A}$ also contains all slope matrices $A(y,x)$ for $x,y\in{\bf{z}}$, then this zero is unique in ${\bf{z}}$.
\end{theorem}

It is to be noted that generally to solve the interval linear system ${\bf{A}}x={\bf{b}}$, for holding the crucial condition (\ref{eq.2.13}) from Theorem \ref{theorem 2.3}, $R$ should be a good approximation for the inverse of ${\textrm{mid}}({\bf{A}})$. The point vector $\tilde{x}$ usually is considered as an approximation to the solution of the midpoint system, i.e., ${\textrm{mid}}({\bf{A}})x={\textrm{mid}}({\bf{b}})$. Also the interval vector ${\bf{x}}$ must be chosen in such a way that (\ref{eq.2.13}) be satisfied. We use $\epsilon$-inflation method \cite{Rum1999} to obtain ${\bf{x}}$.

\subsubsection*{Simultaneous diagonalization assumption.}
Using Theorem \ref{theorem 2.3}, we want to propose a technique for obtaining outer estimations for the solution set of the interval generalized Sylvester matrix equation ${\bf{A}}X{\bf{B}}+{\bf{C}}X{\bf{D}}={\bf{F}}$ provided that $A^{c}$ and $C^{c}$ are simultaneously diagonalizable and also $B^{c}$ and $D^{c}$ are simultaneously diagonalizable, i.e., there are invertible matrices $U$ and $V$ such that
\begin{align}\label{eq.2.14}
A^{c}=UD_{A}U^{-1}, \hspace{.3 cm} C^{c}=UD_{C}U^{-1} \hspace{.6 cm} {\textrm{with $U\in{\mathbb{K}}^{m\times m}$ and $D_{A},D_{C}$ are diagonal}},
\\\label{eq.2.15}
B^{c}=VD_{B}V^{-1}, \hspace{.3 cm} D^{c}=VD_{D}V^{-1} \hspace{.6 cm} {\textrm{with $V\in{\mathbb{K}}^{n\times n}$ and $D_{B},D_{D}$ are diagonal}}.
\end{align}
In fact, the columns of $U$ are eigenvectors of $A^{c}$ and $C^{c}$ and similarly the columns of $V$ are eigenvectors of $B^{c}$ and $D^{c}$. It is worth noting that two matrices are simultaneously diagonalizable if and only if they commute. There are some special and important cases which arise in many practical applications such that this condition holds for them easier and so our approach can be applied for them easily. For some instances consider the interval Lyapunov matrix equation
\begin{equation*}
{\bf{A}}X+X{\bf{A}}^{\top}={\bf{B}}
\end{equation*}
that (taking into account uncertainties) arises in several applications in control theory or the interval Sylvester matrix equation
\begin{equation*}
{\bf{A}}X+X{\bf{B}}={\bf{C}}
\end{equation*}
that (taking into account uncertainties) plays a vital role in many problems such as control theory, image processing and many other applications or the interval version of Kalman-Yakubovich-conjugate matrix equation
\begin{equation*}
X-{\bf{A}}X{\bf{B}}={\bf{C}}
\end{equation*}
that (taking into account uncertainties) plays important roles in theory and applications of stability and control for discrete-time systems. When ${\bf{A}}={\bf{B}}^{\top}$, it is the well-known Stein equation, see \cite{JiaWei2003}. Or the interval continuous-time symmetric Sylvester matrix equation
\begin{equation*}
{\bf{A}}X{\bf{E}}^{\top}+{\bf{E}}X{\bf{A}}^{\top}={\bf{C}},
\end{equation*}
and its discrete-time counterpart
\begin{equation*}
{\bf{A}}X{\bf{A}}^{\top}-{\bf{E}}X{\bf{E}}^{\top}={\bf{C}},
\end{equation*}
which are of particular interest in control theory \cite{GarLau1992}.

\subsubsection*{Preconditioning.}
The pre-multiplication with $U^{-1}$ and post-multiplication with $V$ transform the original system (\ref{eq.1.5}) into a new system
\begin{equation}\label{eq.2.16}
{\bf{A}}_{p}Y{\bf{B}}_{p}+{\bf{C}}_{p}Y{\bf{D}}_{p}={\bf{F}}_{p},
\end{equation}
in which ${\bf{A}}_{p}=U^{-1}{\bf{A}}U$, ${\bf{B}}_{p}=V^{-1}{\bf{B}}V$, ${\bf{C}}_{p}=U^{-1}{\bf{C}}U$, ${\bf{D}}_{p}=V^{-1}{\bf{D}}V$, ${\bf{F}}_{p}=U^{-1}{\bf{F}}V$ and $Y=U^{-1}XV$. It is to be noted that generally for multiplication of three interval matrices ${\bf{G}},{\bf{H}},{\bf{K}}$ of compatible sizes, the associative law does not hold, i.e., $({\bf{G}}{\bf{H}}){\bf{K}}$ is different from ${\bf{G}}({\bf{H}}{\bf{K}})$. But if ${\bf{G}}$ and ${\bf{K}}$ are thin then $({\bf{G}}{\bf{H}}){\bf{K}}={\bf{G}}({\bf{H}}{\bf{K}})$ and so in this case parenthesis can be missed, see \cite{Neu1990}. By this reason, in the formulation of ${\bf{A}}_{p},{\bf{B}}_{p},{\bf{C}}_{p},{\bf{D}}_{p},{\bf{F}}_{p}$ that are constructed by multiplication of three matrices, we can omit the parenthesis. Also all coefficient matrices ${\bf{A}}_{p},{\bf{B}}_{p},{\bf{C}}_{p},{\bf{D}}_{p}$ in (\ref{eq.2.16}) have diagonal midpoints, in fact theoretically we have $A_{p}^{c}=D_{A}$, $B_{p}^{c}=D_{B}$, $C_{p}^{c}=D_{C}$ and $D_{p}^{c}=D_{D}$. So the original system (\ref{eq.1.5}) has been modified to the system (\ref{eq.2.16}) with more tractable coefficient matrices.

We obtain outer estimations for the solution set of the interval generalized Sylvester matrix equation (\ref{eq.1.5}) using the transformed system (\ref{eq.2.16}). First, we use Theorem~ \ref{theorem 2.3} for enclosing the solution set of interval system (\ref{eq.2.16}) and then we will transfer the result enclosure to the original system (\ref{eq.1.5}) using $U$ and $V$. Let $U$, $V$, $D_{A}$, $D_{B}$, $D_{C}$ and $D_{D}$ are numerically computed quantities obtained by a standard method such as Matlab's \verb"eig" function to get the decompositions (\ref{eq.2.14}) and (\ref{eq.2.15}). These quantities are computed numerically which fulfill (\ref{eq.2.14}) and (\ref{eq.2.15}) just approximately. So the midpoint of interval matrices ${\bf{A}}_{p}$, ${\bf{B}}_{p}$, ${\bf{C}}_{p}$ and ${\bf{D}}_{p}$ will not be exactly diagonal. Hence we enclose ${\bf{A}}_{p}$, ${\bf{B}}_{p}$, ${\bf{C}}_{p}$ and ${\bf{D}}_{p}$ by some interval matrices with off-diagonal elements symmetric around zero as follows
\begin{subequations}\label{equation 1}
\begin{align}
 {\bf{A}}_{p}&=U^{-1}{\bf{A}}U \subseteq{\textrm{Diag}}({\textrm{diag}}({\bf{A}}_{p}))+{\textrm{midrad}}(0,\Mag({\textrm{Off}}({\bf{A}}_{p}))),
\\
 {\bf{B}}_{p}&=V^{-1}{\bf{B}}V \subseteq {\textrm{Diag}}({\textrm{diag}}({\bf{B}}_{p}))+{\textrm{midrad}}(0,\Mag({\textrm{Off}}({\bf{B}}_{p}))),
\\
 {\bf{C}}_{p}&=U^{-1}{\bf{C}}U \subseteq {\textrm{Diag}}({\textrm{diag}}({\bf{C}}_{p}))+{\textrm{midrad}}(0,\Mag({\textrm{Off}}({\bf{C}}_{p}))),
\\
 {\bf{D}}_{p}&=V^{-1}{\bf{D}}V \subseteq {\textrm{Diag}}({\textrm{diag}}({\bf{D}}_{p}))+{\textrm{midrad}}(0,\Mag({\textrm{Off}}({\bf{D}}_{p}))),
\end{align}
\end{subequations}
in which ${\textrm{Off}}({\bf{A}}_{p})$ shows the off-diagonal elements of ${\bf{A}}_{p}$, i.e., an interval matrix with zero diagonal entries and its off-diagonal entries are the same as off-diagonal entries of ${\bf{A}}_{p}$. 
So, from now on, we will assume that the midpoints of ${\bf{A}}_{p}$, ${\bf{B}}_{p}$, ${\bf{C}}_{p}$ and ${\bf{D}}_{p}$ are diagonal; otherwise we replace them by their enclosures as in (\ref{equation 1}).

Using Kronecker operator, interval system (\ref{eq.2.16}) is transformed to the following system
\begin{equation}\label{eq.2.17}
{\bf{Q}}_{p}y={\bf{f}}_{p},
\end{equation}
in which ${\bf{Q}}_{p}={\bf{B}}_{p}^{\top}\otimes {\bf{A}}_{p}+{\bf{D}}_{p}^{\top}\otimes{\bf{C}}_{p}$, $y={\vect}(Y)$ and  ${\bf{f}}_{p}={\vect}({\bf{F}}_{p})$. We should determine an approximate inverse for ${\textrm{mid}}({\bf{Q}}_{p})$. Using (\ref{eq.2.9}) we know that \begin{equation*}
{\textrm{mid}}({\bf{Q}}_{p})=B_{p}^{\top c}\otimes A_{p}^{ c}+ D_{p}^{\top c} \otimes C_{p}^{ c} =(V^{-1}B^{c}V)^{\top}\otimes(U^{-1}A^{c}U)+(V^{-1}D^{c}V)^{\top}\otimes(U^{-1}C^{c}U).
\end{equation*}
On the other hand because of the previously mentioned reasons, we know that $U^{-1}A^{c}U$, $V^{-1}B^{c}V$, $U^{-1}C^{c}U$ and $V^{-1}D^{c}V$ will not be exactly diagonal but we can expect them to be very close to $D_{A}$, $D_{B}$, $D_{C}$ and $D_{D}$, respectively. So if we define
\begin{equation*}
\Lambda=D_{B}\otimes D_{A}+D_{D}\otimes D_{C}.
\end{equation*}
Then we can expect that the diagonal matrix $\Lambda$ is a good approximation for $Q_{p}^{c}$ and so 
\begin{equation}\label{dfR}
R=\Lambda^{-1}
\end{equation}
to be a good approximate inverse for $Q_{p}^{c}$.

Let $A_{p}\in{\bf{A}}_{p}$, $B_{p}\in{\bf{B}}_{p}$, $C_{p}\in{\bf{C}}_{p}$, $D_{p}\in{\bf{D}}_{p}$ and $F_{p}\in{\bf{F}}_{p}$ be arbitrary point matrices. Put $f(y)=Q_{p}y-f_{p}$ in which $Q_{p}=B_{p}^{\top}\otimes A_{p}+D_{p}^{\top}\otimes C_{p}$ and $f_{p}={\vect}(F_{p})$. Now, for using Theorem~\ref{theorem 2.3} we should obtain an interval vector containing the set
\begin{equation}\label{eq.2.18}
\mathcal{H}=\{ -R(Q_{p}\tilde{x}-f_{p})+(I_{mn}-RQ_{p})x: A_{p}\in{\bf{A}}_{p}, B_{p}\in{\bf{B}}_{p}, C_{p}\in{\bf{C}}_{p}, D_{p}\in{\bf{D}}_{p}, F_{p}\in{\bf{F}}_{p}, x\in{\bf{x}}\}.
\end{equation}
We do this work by separately computing enclosures for the two following sets
\begin{equation}\label{eq.2.19}
\mathcal{M}=\{ -R(Q_{p}\tilde{x}-f_{p}): A_{p}\in{\bf{A}}_{p}, B_{p}\in{\bf{B}}_{p}, C_{p}\in{\bf{C}}_{p}, D_{p}\in{\bf{D}}_{p}, F_{p}\in{\bf{F}}_{p}\},
\end{equation}
\begin{equation}\label{eq.2.20}
\mathcal{N}=\{ (I_{mn}-RQ_{p})x: A_{p}\in{\bf{A}}_{p}, B_{p}\in{\bf{B}}_{p}, C_{p}\in{\bf{C}}_{p}, D_{p}\in{\bf{D}}_{p},  x\in{\bf{x}}\}.
\end{equation}

\begin{lemma}\label{lemma 2.3}
Consider the interval system (\ref{eq.2.16}). Let 
$\tilde{X}\in{\mathbb{K}}^{mn\times mn}$ and ${\bf{X}}\in{\mathbb{IK}}^{mn\times mn}$ with $\tilde{x}=\vect({\tilde{X}})$ and ${\bf{x}}=\vect({\bf{X}})$ are given. Using the introduced notation before, define
\begin{equation}\label{eq.2.23}
{\bf{M}}=({\bf{F}}_{p}-({\bf{A}}_{p}\tilde{X}){\bf{B}}_{p}-({\bf{C}}_{p}\tilde{X}){\bf{D}}_{p})./S,
\end{equation}
\begin{equation}\label{eq.2.28}
{\bf{N}}=\big<0,\big(A_{p}^{\Delta} X^{\Delta}|B_{p}^{c}|+ \Mag({\bf{A}}_{p})X^{\Delta}B_{p}^{\Delta} + C_{p}^{\Delta} X^{\Delta}|D_{p}^{c}|+ \Mag({\bf{C}}_{p})X^{\Delta}D_{p}^{\Delta} \big)./|S| \big>,
\end{equation}
\begin{equation}\label{eq.2.30}
{\bf{H}}={\bf{M}}+{\bf{N}},
\end{equation}
where  $S$ is the $m$-by-$n$ matrix defined by
\begin{equation*}
\Lambda= {\textrm{Diag}}({\vect}(S)).
\end{equation*}
Then for the set $\mathcal{H}$ defined by (\ref{eq.2.18}) we have
\begin{equation}\label{eq.2.29}
\mathcal{H}\subseteq {\vect}({\bf{H}}).
\end{equation}
\end{lemma}

\begin{proof}
First, for enclosing $\mathcal{M}$ in (\ref{eq.2.19}), for any arbitrary $A_{p}\in{\bf{A}}_{p}$, $B_{p}\in{\bf{B}}_{p}$, $C_{p}\in{\bf{C}}_{p}$, $D_{p}\in{\bf{D}}_{p}$, $F_{p}\in{\bf{F}}_{p}$ and a given vector $\tilde{x}=\vect(\tilde{X})\in\mathbb{K}^{mn}$, using Lemma \ref{lemma 2.1} and $R$ from (\ref{dfR}) we have
{\setlength\arraycolsep{2pt}
\begin{eqnarray}\label{eq.2.21}
&& -R(Q_{p}\tilde{x}-f_{p})= -R[(B_{p}^{\top}\otimes A_{p}+D_{p}^{\top}\otimes C_{p}){\vect}(\tilde{X})-{\vect}(F_{p})]
\nonumber\\
&& =-R[{\vect}(A_{p}\tilde{X}B_{p})+{\vect}(C_{p}\tilde{X}D_{p})-{\vect}(F_{p})]=-R{\vect}(A_{p}\tilde{X}B_{p}+C_{p}\tilde{X}D_{p} -F_{p})
\nonumber\\
&& = {\vect}\big( (F_{p}-A_{p}\tilde{X}B_{p}-C_{p}\tilde{X}D_{p})./S \big),
\end{eqnarray}}
By (\ref{equation 1}), since $A_{p}\in{\bf{A}}_{p}$, $B_{p}\in{\bf{B}}_{p}$, $C_{p}\in{\bf{C}}_{p}$ and $D_{p}\in{\bf{D}}_{p}$, if we define
\begin{equation*}
{\bf{M}}=({\bf{F}}_{p}-({\bf{A}}_{p}\tilde{X}){\bf{B}}_{p}-({\bf{C}}_{p}\tilde{X}){\bf{D}}_{p})./S,
\end{equation*}
then using (\ref{eq.2.21}) and due to enclosure property, we obtain $\mathcal{M}\subseteq {\vect}({\bf{M}})$.

Second, for enclosing $\mathcal{N}$ defined by (\ref{eq.2.20}), it is obvious that
\begin{equation}\label{eq.2.24}
\mathcal{N}\subseteq [I_{mn}-R({\bf{B}}_{p}^{\top}\otimes {\bf{A}}_{p}+{\bf{D}}_{p}^{\top}\otimes {\bf{C}}_{p})]{\bf{x}}.
\end{equation}
In Krawczyk type methods (Theorem \ref{theorem 2.2}), in each iterate of the Krawczyk operator, ${\bf{x}}$ can be taken to be symmetric. Here also we let ${\bf{x}}$ to be symmetric, i.e., ${\bf{x}}=\big< 0,x^{\Delta} \big>$. We want to determine $(I_{mn}-R{\bf{Q}}_{p}){\bf{x}}$ in which ${\bf{Q}}_{p}={\bf{B}}_{p}^{\top}\otimes {\bf{A}}_{p}+{\bf{D}}_{p}^{\top}\otimes{\bf{C}}_{p}$. Let ${\bf{Q}}_{p}=\big< Q_{p}^{c},Q_{p}^{\Delta} \big>$, so we have
\begin{align*}
(R{\bf{Q}}_{p})_{ij}
&= \sum_{k=1}^{mn} R_{ik}{\bf{Q}}_{pkj} 
 = \sum_{k=1}^{mn}\big< R_{ik},0 \big> \big< Q_{pkj}^{c},Q_{pkj}^{\Delta} \big>
\nonumber\\
& = \sum_{k=1}^{mn}\big< R_{ik}Q_{pkj}^{c},|R_{ik}|Q_{pkj}^{\Delta} \big>
= \left< \sum_{k=1}^{mn}R_{ik}Q_{pkj}^{c},\sum_{k=1}^{mn}|R_{ik}|Q_{pkj}^{\Delta} \right>
\nonumber\\
& = \big< (RQ_{p}^{c})_{ij}, (|R|Q_{p}^{\Delta})_{ij} \big> ,
\end{align*}
and thus $R{\bf{Q}}_{p}=\big< RQ_{p}^{c}, |R|Q_{p}^{\Delta} \big>=\big< I_{mn}, |R|Q_{p}^{\Delta} \big>$. Hence
\begin{equation*}
I_{mn}-R{\bf{Q}}_{p}=\big< I_{mn}, 0 \big>-\big< I_{mn}, |R|Q_{p}^{\Delta} \big>=\big<0, |R|Q_{p}^{\Delta} \big>.
\end{equation*}
So if we set ${\bf{W}}=\big<0, |R|Q_{p}^{\Delta} \big>$ then
\begin{align*}
 ({\bf{W}}{\bf{x}})_{i}&= \sum_{k=1}^{mn} {\bf{W}}_{ik}{\bf{x}}_{k} 
= \sum_{k=1}^{mn} \big<0, (|R|Q_{p}^{\Delta})_{ik} \big> \big<0, x_{k}^{\Delta} \big>
\nonumber\\
& = \sum_{k=1}^{mn} \big<0, (|R|Q_{p}^{\Delta})_{ik}x_{k}^{\Delta}  \big> =\left<0, \sum_{k=1}^{mn} (|R|Q_{p}^{\Delta})_{ik}x_{k}^{\Delta}  \right>
\nonumber\\
& = \big<0, (|R|Q_{p}^{\Delta}x^{\Delta})_{i}  \big>,
\end{align*}
and so
\begin{equation}\label{eq.2.25}
(I_{mn}-R{\bf{Q}}_{p}){\bf{x}}={\textrm{midrad}}(0,|R|Q_{p}^{\Delta}x^{\Delta}).
\end{equation}
Since ${\bf{Q}}_{p}={\bf{B}}_{p}^{\top}\otimes {\bf{A}}_{p}+{\bf{D}}_{p}^{\top}\otimes{\bf{C}}_{p}$, using (\ref{eq.2.24}) and (\ref{eq.2.25}) we obtain
\begin{equation}\label{eq.2.26}
{\mathcal{N}}\subseteq{\textrm{midrad}}(0,|R|Q_{p}^{\Delta}x^{\Delta}).
\end{equation}
By (\ref{eq.2.9}), we get $Q_{p}^{\Delta}=|B_{p}^{\top c}|\otimes A_{p}^{\Delta} + B_{p}^{\top \Delta}\otimes\Mag({\bf{A}}_{p})+   |D_{p}^{\top c}|\otimes C_{p}^{\Delta}+ D_{p}^{\top \Delta}\otimes\Mag({\bf{C}}_{p}) $. So if $x^{\Delta}={\vect}(X^{\Delta})$, then we have
{\setlength\arraycolsep{2pt}
\begin{eqnarray*}
 |R|Q_{p}^{\Delta}x^{\Delta}=&& |R| [ |B_{p}^{\top c}|\otimes A_{p}^{\Delta} + B_{p}^{\top \Delta}\otimes \Mag({\bf{A}}_{p})+
\nonumber\\
&&  |D_{p}^{\top c}|\otimes C_{p}^{\Delta}+ D_{p}^{\top \Delta}\otimes \Mag({\bf{C}}_{p})]{\vect}(X^{\Delta}),
\end{eqnarray*}}
which by using Lemma \ref{lemma 2.1} is equivalent to
{\setlength\arraycolsep{2pt}
\begin{eqnarray*}
|R|Q_{p}^{\Delta}x^{\Delta} &&= |R|{\vect} \big( A_{p}^{\Delta} X^{\Delta}|B_{p}^{c}|+ \Mag({\bf{A}}_{p})X^{\Delta}B_{p}^{\Delta} +       C_{p}^{\Delta} X^{\Delta}|D_{p}^{c}|+\Mag({\bf{C}}_{p})X^{\Delta}D_{p}^{\Delta}  \big)
\nonumber\\
&&= {\vect} \Big(  \big(A_{p}^{\Delta} X^{\Delta}|B_{p}^{c}|+ \Mag({\bf{A}}_{p})X^{\Delta}B_{p}^{\Delta} + C_{p}^{\Delta} X^{\Delta}|D_{p}^{c}|+\Mag({\bf{C}}_{p})X^{\Delta}D_{p}^{\Delta} \big)./|S| \Big).
\end{eqnarray*}}
So if we put
\begin{equation*}
{\bf{N}}=\big<0,\big(A_{p}^{\Delta} X^{\Delta}|B_{p}^{c}|+ \Mag({\bf{A}}_{p})X^{\Delta}B_{p}^{\Delta} + C_{p}^{\Delta} X^{\Delta}|D_{p}^{c}|+ \Mag({\bf{C}}_{p})X^{\Delta}D_{p}^{\Delta} \big)./|S| \big>,
\end{equation*}
then by (\ref{eq.2.26}) we conclude that $\mathcal{N}\subseteq {\vect}({\bf{N}})$.\\
If the interval matrices ${\bf{M}}$ and ${\bf{N}}$ are constructed by (\ref{eq.2.23}) and (\ref{eq.2.28}) respectively, then for the set $\mathcal{H}$ defined by (\ref{eq.2.18}) we will have
\begin{equation*}
\mathcal{H}\subseteq {\vect}({\bf{H}}),
\end{equation*}
in which
\begin{equation*}
{\bf{H}}={\bf{M}}+{\bf{N}}.\qedhere
\end{equation*}
\end{proof}

By the argument leading to construction of the interval matrices ${\bf{M}}$, ${\bf{N}}$ and ${\bf{H}}$ defined by (\ref{eq.2.23}), (\ref{eq.2.28}) and (\ref{eq.2.30}), respectively, we can represent the following theorem. Its proof in some parts is similar to the proof of Theorem 3.4 in \cite{HasDeh2012}. 

\begin{theorem}\label{theorem 2.4}
Consider the interval generalized Sylvester matrix equation (\ref{eq.1.5}) and let $R\in{\mathbb{K}}^{mn\times mn}$, $\tilde{X}\in{\mathbb{K}}^{m\times n}$ and ${\bf{X}}\in{\mathbb{IK}}^{m\times n}$ are given. If the interval matrix ${\bf{H}}$ obtained by (\ref{eq.2.30}) satisfies ${\bf{H}}\subseteq {\textrm{int}}({\bf{X}})$ then $R$ is nonsingular and for every $A\in{\bf{A}}$, $B\in{\bf{B}}$, $C\in{\bf{C}}$, $D\in{\bf{D}}$ and $F\in{\bf{F}}$ the corresponding generalized Sylvester matrix equation $AXB+CXD=F$ has a unique solution $\hat{X}\in U(\tilde{X}+{\bf{X}})V^{-1}$. Therefore $\Xi\subseteq U(\tilde{X}+{\bf{X}})V^{-1}$.
\end{theorem}

\emph{Remark.} Notice that the enclosure of $\Xi$ is of the form $U(\tilde{X}+{\bf{X}})V^{-1}$, which may be convenient for further processing (similarly as affine interval arithmetic) contrary to simple interval evaluation resulting to an interval matrix.

\begin{proof}
Suppose $A\in{\bf{A}}$, $B\in{\bf{B}}$, $C\in{\bf{C}}$, $D\in{\bf{D}}$ and $F\in{\bf{F}}$ and define $A_{p}=U^{-1}AU$, $B_{p}=V^{-1}BV$, $C_{p}=U^{-1}CU$, $D_{p}=V^{-1}DV$,  $F_{p}=U^{-1}FV$ and $Q_{p}=B_{p}^{\top}\otimes A_{p}+D_{p}^{\top}\otimes C_{p}$, where $U$ and $V$ are eigenvector matrices introduced by (\ref{eq.2.14}) and (\ref{eq.2.15}), respectively. By fundamental enclosure property of circular interval arithmetic and using relation (\ref{eq.2.18}) and the result of Lemma \ref{lemma 2.3} we can write
\begin{equation}\label{eq.2.31}
\{ R(f_{p}-Q_{p}\tilde{x})+(I_{mn}-RQ_{p})x:  x\in{\bf{x}}\}\subseteq {\vect}({\bf{H}}),
\end{equation}
in which $f_{p}={\vect}(F_{p})$, $\tilde{x}={\vect}(\tilde{X})$ and ${\bf{x}}={\vect}({\bf{X}})$. Since ${\bf{H}}\subseteq{\textrm{int}}({\bf{X}})$, by (\ref{eq.2.31}) we conclude that
\begin{equation}\label{eq.2.32}
\{ R(f_{p}-Q_{p}\tilde{x})+(I_{mn}-RQ_{p})x:  x\in{\bf{x}}\}\subseteq {\textrm{int}}({\bf{x}}).
\end{equation}
Now, we define the function $g:\mathbb{K}^{mn}\rightarrow {\mathbb{K}^{mn}}$ as follows
\begin{equation*}
g(x)= R(f_{p}-Q_{p}\tilde{x})+(I_{mn}-RQ_{p})x=x+R[f_{p}-Q_{p}(\tilde{x}+x)].
\end{equation*}
According to (\ref{eq.2.32}), the continuous function $g$ maps the compact convex set ${\bf{x}}$ into itself and so Brouwer's fixed point theorem implies that there exists $x^{*}\in{\bf{x}}$ such that $g(x^{*})=x^{*}$ and hence
\begin{equation}\label{eq.2.33}
R[f_{p}-Q_{p}(\tilde{x}+x^{*})]=0.
\end{equation}
Now, we show that $R$ and $Q_{p}$ are nonsingular. If $R$ or $Q_{p}$ are singular then there exists a nonzero vector $v\neq 0$ such that $RQ_{p}v=0$. By definition of the function $g$, $x^{*}+\alpha v$  for all $\alpha\in\mathbb{K}$  would be a fixed point of $g$. On the other hand it is obvious that there would be some $\alpha^{*}\in\mathbb{K}$ such that $x^{*}+\alpha^{*} v\in \partial{\bf{x}}$ which is a contradiction to (\ref{eq.2.32}). So $R$ and $Q_{p}$ are nonsingular. Because $Q_{p}$ is nonsingular so the linear system $Q_{p}y=f_{p}$ in which $y={\vect}(U^{-1}XV)$  is uniquely solvable. But this equation is equivalent to the equation $AXB+CXD=F$ (in vectorization form) and so the uniquely solvability of this equation is concluded. Because $R$ is nonsingular so by (\ref{eq.2.33}) we conclude that $Q_{p}(\tilde{x}+x^{*})=f_{p}$ and hence the system $Q_{p}y=f_{p}$ has a unique solution $\hat{y}=\tilde{x}+x^{*}\in\tilde{x}+{\bf{x}}$. Considering the relation between two systems $Q_{p}y=f_{p}$ and $Qx=f$ in which $Q=B^{\top}\otimes A+D^{\top}\otimes C$, $x=(V^{-\top}\otimes U)y$ and $f={\vect}(F)$ shows that  $Qx=f$ has a unique solution $\hat{x}\in(V^{-\top}\otimes U)(\hat{x}+\bf{x})$. So the generalized Sylvester matrix equation $AXB+CXD=F$ has a unique solution $\hat{X}\in U(\tilde{X}+{\bf{X}})V^{-1}$ with $\tilde{x}={\vect}(\tilde{X})$. Since $A$, $B$, $C$, $D$ and $F$ were chosen arbitrary, we conclude that $\Xi\subseteq U(\tilde{X}+{\bf{X}})V^{-1}$.
\end{proof}

Let us point out that in the Krawczyk type operators for solving an interval linear system ${\bf{A}}x={\bf{b}}$, usually $\tilde{x}$ is taken as an approximate solution for the midpoint system, i.e., $A^{c}x=b^{c}$. Here, also we follow this and choose $\tilde{x}$ as an approximate solution to the midpoint system of ${\bf{Q_{p}}}x={\bf{f_{p}}}$, i.e., $Q_{p}^{c}x=f_{p}^{c}$. But by (\ref{eq.2.9}) we know that $Q_{p}^{c}= B_{p}^{\top c}\otimes A_{p}^{ c}+ D_{p}^{\top c} \otimes C_{p}^{ c}$. On the other hand, theoretically we have $A_{p}^{c}=D_{A}$, $B_{p}^{c}=D_{B}$, $C_{p}^{c}=D_{C}$ and $D_{p}^{c}=D_{D}$ and so $Q_{p}^{c}$ is a diagonal matrix. In fact using the notations introduced before, $Q_{p}^{c}=\Lambda$ and $(Q_{p}^{c})^{-1}=R$. Thus $\tilde{x}=Rf_{p}^{c}=(Diag({\vect}(S)))^{-1}{\vect}(F_{p}^{c})={\vect}(F_{p}^{c}./S)$. Of course in our approach, we deal with the large scale equation $Q_{p}^{c}x=f_{p}^{c}$ implicitly and we use $\tilde{X}=F_{p}^{c}./S$ instead of $\tilde{x}$. Algorithm 1 describes our approach for obtaining an outer estimation for the solution set of the interval generalized Sylvester matrix equation (\ref{eq.1.5}). In this algorithm we use standard $\epsilon$-inflation method to obtain ${\bf{X}}$ and $\epsilon$ is the machine precision. The obtained enclosure ${\bf{Z}}$ by Algorithm \ref{algorithm 1} has this advantage that it is in the form of ${\bf{Z}}=U{\bf{Y}}V^{-1}$ in which ${\bf{Y}}=\tilde{X}+{\bf{X}}$.

\begin{algorithm}[ht]
\caption{A modified Krawczyk method for enclosing the solution set of the interval system (\ref{eq.1.5})}
\begin{algorithmic}[1]
\Require
\Statex Matrices ${\bf{A}},{\bf{C}}\in\mathbb{IK}^{m\times m}$, ${\bf{B}},{\bf{D}}\in\mathbb{IK}^{n\times n}$ and ${\bf{F}}\in\mathbb{IK}^{m\times n}$.
\Ensure
\Statex Either an interval matrix ${\bf{Z}}$ containing the united solution set $\Xi$ of the interval generalized Sylvester matrix equation (\ref{eq.1.5}) or the report "Method can not obtain outer estimation".
\Statex
\State Use a floating point algorithm for computing $U$, $V$, $D_{A}$, $D_{B}$, $D_{C}$ and $D_{D}$ in spectral decompositions (\ref{eq.2.14}) and (\ref{eq.2.15}).
\State Compute ${\bf{A}}_{p}$, ${\bf{B}}_{p}$, ${\bf{C}}_{p}$ and ${\bf{D}}_{p}$ by (\ref{equation 1}) using interval arithmetic.
\State Put $S=[S_{1}|\ldots|S_{n}]\in\mathbb{K}^{m\times n}$ such that $S_{j}=D_{B}(j,j)*{\textrm{diag}}(D_{A})+D_{D}(j,j)*{\textrm{diag}}(D_{C})$.
\State Compute $\tilde{X}={\textrm{mid}}({\bf{F}}_{p})./S$.
\State Compute  ${\bf{M}}=({\bf{F}}_{p}-({\bf{A}}_{p}\tilde{X}){\bf{B}}_{p}-({\bf{C}}_{p}\tilde{X}){\bf{D}}_{p})./S$.
\State Put ${\bf{H}}={\bf{M}}$; ready=0; $k=0$; $kmax=15$;
\State $E=0.1 *{\textrm{rad}}({\bf{M}})*[-1,1]+{\textrm{midrad}}(0,10*\epsilon)$
\While{ ($\sim$ ready \& $k<kmax$)}
\State Put ${\bf{X}}={\bf{H}}+E$   \Comment \{$\epsilon$-inflation\}
\State Put $k=k+1$
\State Compute ${\bf{N}}$ for input ${\bf{X}}$ by (\ref{eq.2.28})
\State Put ${\bf{H}}={\bf{M}}+{\bf{N}}$
\State ready=in0({\bf{H}},{\bf{X}})  \Comment \{Checking condition (\ref{eq.2.13})\}
\EndWhile
\If {ready}
\State Output ${\bf{Z}}=U(\tilde{X}+{\bf{X}})V^{-1}$
\Else
\State Output "Method can not obtain outer estimation"
\EndIf
\end{algorithmic}
\label{algorithm 1}
\end{algorithm}

\begin{theorem}\label{theorem 2.5}
Algorithm \ref{algorithm 1} requires $O(m^{3}+n^{3})$ arithmetic operations per iteration.
\end{theorem}
\begin{proof}
The cost of spectral decompositions in line 1 is cubic with respect to the dimension of the involved matrices and adding together yields  $O(m^{3}+n^{3})$. Other computational parts of Algorithm \ref{algorithm 1} contain vector-vector or matrix-matrix operations including multiplications or additions between $m$-dimensional vectors or matrices of dimensions $m\times m$, $n\times n$ or $m\times n$ and hence their cost is at most $O(m^{3}+n^{3})$. So gathering all computational costs yield Algorithm~ \ref{algorithm 1} requires $O(m^{3}+n^{3})$ arithmetic operations.
\end{proof}

\subsection{Theory based on block diagonalization}\label{subsection 2.3}
The introduced theory in Subsection \ref{subsection 2.2} is applicable when the condition ${\bf{H}}\subseteq {\textrm{int}}({\bf{X}})$ in Theorem \ref{theorem 2.4} is likely to take place. But if at least one of the $A^{c}$, $B^{c}$, $C^{c}$ and $D^{c}$ are not diagonalizable or if the eigenvector matrices $U$ or $V$ are ill-conditioned then the radii of various computed matrices in Algorithm \ref{algorithm 1} will become very large and so the condition ${\bf{H}}\subseteq {\textrm{int}}({\bf{X}})$ will not be satisfied. In such situation, we can utilize block diagonalization introduced by Bavely and Stewart \cite{BavSte1979}. Again we suppose that $A^{c}$ and $C^{c}$ are simultaneously block diagonalizable and also $B^{c}$ and $D^{c}$ are simultaneously block diagonalizable, i.e., there are invertible matrices $U\in{\mathbb{K}}^{m\times m}$ and $V\in{\mathbb{K}}^{n\times n}$ such that

\begin{equation}\label{eq.2.34}
A^{c}=UD_{A}U^{-1}, \hspace{.3 cm} B^{c}=VD_{B}V^{-1}, \hspace{.3 cm} C^{c}=UD_{C}U^{-1}, \hspace{.3 cm} D^{c}=VD_{D}V^{-1} ,
\end{equation}
in which $D_{A}$, $D_{B}$, $D_{C}$ and $D_{D}$ are block  diagonal with each diagonal block being triangular. Also we suppose that the size of diagonal blocks of $D_{B}$ and $D_{D}$ are the same pairwise. Authors in \cite{FroHas2012} used block diagonalization skillfully for computing verified solutions of the real-valued Sylvester matrix equation $AX+XB=C$. Here we utilize a similar approach for our problem. Although we can require either upper or lower triangular form in principle, but from a computational point of view it is an advantage that we assume that $D_{A}$ and $D_{C}$ are upper triangular and $D_{B}$ and $D_{D}$ to be lower triangular, since with this assumption $D_{B}^{\top}\otimes D_{A}+D_{D}^{\top}\otimes D_{C}$ will be also triangular. The intended block diagonalization algorithm allows to trade a better condition of $U$ for larger diagonal blocks in $D_{A}$.

Considering block diagonalizations (\ref{eq.2.34}) instead of spectral decompositions (\ref{eq.2.14}) and (\ref{eq.2.15}) yields an extension of the introduced approach in Subsection \ref{subsection 2.2}. Here based on the block diagonalizations (\ref{eq.2.34}), we must restore
\begin{equation*}
\Lambda=D_{B}^{\top}\otimes D_{A}+D_{D}^{\top}\otimes D_{C}.
\end{equation*}
It is obvious that $\Lambda$ is not diagonal any more. For describing structure of $\Lambda$ in more details let
\begin{equation*}
D_{A}={\textrm{Diag}}(D_{A}^{1},\ldots,D_{A}^{W}),  \hspace{1 cm} D_{C}={\textrm{Diag}}(D_{C}^{1},\ldots,D_{C}^{V}),
\end{equation*}
\begin{equation*}
D_{B}={\textrm{Diag}}(D_{B}^{1},\ldots,D_{B}^{L}),  \hspace{1 cm} D_{D}={\textrm{Diag}}(D_{D}^{1},\ldots,D_{D}^{L}),
\end{equation*}
where $D_{A}^{j}\in\mathbb{K}^{\omega_{j}\times\omega_{j}}$ for $j=1,\ldots,W$, and $D_{C}^{j}\in\mathbb{K}^{\nu_{j}\times\nu_{j}}$ for $j=1,\ldots,V$, are upper triangular diagonal blocks of $D_{A}$ and $D_{C}$, respectively. Also $D_{B}^{j}\in\mathbb{K}^{{\ell}_{j}\times{\ell}_{j}}$ and $D_{D}^{j}\in\mathbb{K}^{{\ell}_{j}\times{\ell}_{j}}$ for $j=1,\ldots,L$, are lower triangular diagonal blocks of $D_{B}$ and $D_{D}$, respectively. Then
{\setlength\arraycolsep{2pt}
\begin{eqnarray}\label{eq.2.35}
 \Lambda= && {\textrm{Diag}}(\Lambda_{1},\ldots,\Lambda_{L}),
\nonumber\\
&& \Lambda_{j}=(D_{B}^{j})^{\top}\otimes D_{A}+(D_{D}^{j})^{\top}\otimes D_{C}\in\mathbb{K}^{m{\ell}_{j}\times m{\ell}_{j}},   \hspace{1 cm} j=1,\ldots,L,
\end{eqnarray}}
in which $\Lambda_{j}$, $j=1,\ldots,L$, are upper triangular diagonal blocks that each of them has internal sub-block structure as follows
\begin{equation*}
\Lambda_{j}=\left(
  \begin{array}{cccc}
    (D_{B}^{j})_{11} D_{A}+(D_{D}^{j})_{11}D_{C} & (D_{B}^{j})_{21} D_{A}+(D_{D}^{j})_{21}D_{C} & \ldots & (D_{B}^{j})_{{\ell}_{j}1} D_{A}+(D_{D}^{j})_{{\ell}_{j}1}D_{C} \\
    0 & (D_{B}^{j})_{22} D_{A}+(D_{D}^{j})_{22}D_{C} & \ldots & (D_{B}^{j})_{{\ell}_{j}2} D_{A}+(D_{D}^{j})_{{\ell}_{j}2}D_{C} \\
    0 & 0 & \ddots & \vdots \\
    0 & 0 & \ldots & (D_{B}^{j})_{{\ell}_{j}{\ell}_{j}} D_{A}+(D_{D}^{j})_{{\ell}_{j}{\ell}_{j}}D_{C} \\
  \end{array}
\right).
\end{equation*}

The inverse of $\Lambda$ defined by (\ref{eq.2.35}) should be replaced by $R$ in Subsection \ref{subsection 2.2}. By this substitution, it is obvious that Theorem \ref{theorem 2.4} holds still providing we regulate the computation of ${\bf{M}}$ and ${\bf{N}}$, respectively by (\ref{eq.2.23}) and (\ref{eq.2.28}) with this fact that $\Lambda$ is not diagonal any more and now is block diagonal with upper triangular diagonal blocks. We can do this without computing $R=\Lambda^{-1}$ explicitly, using backward substitution. Note that due to the block structure of $\Lambda$, this backward substitution breaks into several parts, each part deals with $\Lambda_{j}$ from (\ref{eq.2.35}).

On the other hand, Algorithm~\ref{algorithm 1} must be modified also in four points. First in line 1, the computation of spectral decompositions must be replaced by the computation of block diagonalizations (\ref{eq.2.34}). In line 4, instead of computation of $\tilde{X}$ by pointwise divisions we perform backward substitution on $\Lambda$ with the right-hand side vector ${\vect}({\textrm{mid}}({\bf{F}}_{p}))$ that yields a vector $\tilde{x}$ which we reduce it to a matrix $\tilde{X}$ with $\tilde{x}={\vect}(\tilde{X})$ to continue. Similar modifications are needed for lines 5 and 11. The modified algorithm consumes more time than the Algorithm \ref{algorithm 1} in which we only have matrix-matrix operations, of course it is still computationally efficient. In fact, if we define $u=\max\{\ell_{1},\ldots,\ell_{L}\}$,  $l=\min\{\ell_{1},\ldots,\ell_{L}\}$ and $\lambda=u^{2}/l$ then the computational complexity of the backward substitution process in modified algorithm is at most $O(\lambda m^{2}n)$ which is less than $O(\lambda (m^{3}+n^{3}))$ and since the complexity of the remaining computations in the modified algorithm is  at most of order $O(m^{3}+n^{3})$, we conclude that the modified algorithm is still has a cost of $O(m^{3}+n^{3})$ arithmetic operations. It is to be noted that due to the dependency problem in interval arithmetic, backward substitution for line 5 of Algorithm \ref{algorithm 1} may causes an interval matrix ${\bf{M}}$ which is substantially large and so condition (\ref{eq.2.13}) may not be held any more. Unless we have lucky sign constellations in the components of $\Lambda$ and ${\bf{F}}_{p}-({\bf{A}}_{p}\tilde{X}){\bf{B}}_{p}-({\bf{C}}_{p}\tilde{X}){\bf{D}}_{p}$. To obtain more insight into this subject, see the end of Section~3.2 of \cite{FroHas2012}.

\subsection{An iterative method for enclosing the solution set}\label{subsection 2.4}
In this subsection, we present an iterative technique for obtaining outer estimations for the solution set of the interval generalized Sylvester matrix equation (\ref{eq.1.5}). The approach is based on the spectral decompositions (\ref{eq.2.14}) and (\ref{eq.2.15}). Similar to Subsection \ref{subsection 2.2}, we transform the interval system (\ref{eq.1.5}) into the following system
\begin{equation}\label{eq. 2.39}
{\bf{A}}_{p}Y{\bf{B}}_{p}+{\bf{C}}_{p}Y{\bf{D}}_{p}={\bf{F}}_{p},
\end{equation}
by pre-multiplication with $U^{-1}$ and post-multiplication with $V$. In equation (\ref{eq. 2.39}) we have ${\bf{A}}_{p}=U^{-1}{\bf{A}}U$, ${\bf{B}}_{p}=V^{-1}{\bf{B}}V$, ${\bf{C}}_{p}=U^{-1}{\bf{C}}U$, ${\bf{D}}_{p}=V^{-1}{\bf{D}}V$, ${\bf{F}}_{p}=U^{-1}{\bf{F}}V$ and $Y=U^{-1}XV$.

And suppose vectors $a\in\mathbb{K}^{m}$, $b\in\mathbb{K}^{n}$, $c\in\mathbb{K}^{m}$ and $d\in\mathbb{K}^{n}$ are such that
\begin{equation}\label{eq. 2.41}
A_{p}^{c}={\textrm{Diag}}(a),   \hspace{.6 cm}  B_{p}^{c}={\textrm{Diag}}(b),   \hspace{.6 cm}  C_{p}^{c}={\textrm{Diag}}(c),   \hspace{.6 cm}  D_{p}^{c}={\textrm{Diag}}(d).
\end{equation}

Suppose $Y$ belongs to the solution set of the interval system (\ref{eq. 2.39}) and let an initial interval matrix ${\bf{Y}}\in\mathbb{IK}^{m\times n}$ that is an enclosure for the solution set of (\ref{eq. 2.39}) be given. Using Proposition \ref{proposition 2.1}, we can write
\begin{equation}\label{eq. 2.42}
|A_{p}^{c}YB_{p}^{c}+C_{p}^{c}YD_{p}^{c}-F_{p}^{c}|\leq \Mag({\bf{A}}_{p})|Y|B_{p}^{\Delta}+  A_{p}^{\Delta}|Y B_{p}^{c}|  +\Mag({\bf{C}}_{p})|Y|D_{p}^{\Delta}+  C_{p}^{\Delta}|Y D_{p}^{c}|+ F_{p}^{\Delta}.
\end{equation}
Since $Y\in{\bf{Y}}$ so the right-hand side of inequality (\ref{eq. 2.42}) is less than $T$ defined as follows
\begin{equation*}
T:= \Mag({\bf{A}}_{p})|{\bf{Y}}|B_{p}^{\Delta}+  A_{p}^{\Delta}|{\bf{Y}} B_{p}^{c}|  +\Mag({\bf{C}}_{p})|{\bf{Y}}|D_{p}^{\Delta}+  C_{p}^{\Delta}|{\bf{Y}} D_{p}^{c}|+ F_{p}^{\Delta}.
\end{equation*}
Thus we conclude that
\begin{equation*}
|A_{p}^{c}YB_{p}^{c}+C_{p}^{c}YD_{p}^{c}-F_{p}^{c}|\leq T.
\end{equation*}
On the other hand by (\ref{eq. 2.41}) and a simple computation we obtain $A_{p}^{c}YB_{p}^{c}+C_{p}^{c}YD_{p}^{c}=Y \circ(ab^{\top}+cd^{\top})$ where $\circ$ stands for the Hadamard product. And so the last inequality is equivalent to
\begin{equation*}
|Y \circ(ab^{\top}+cd^{\top})-F_{p}^{c}|\leq T,
\end{equation*}
that yields
\begin{equation*}
Y \circ(ab^{\top}+cd^{\top})\in [F_{p}^{c}- T, F_{p}^{c}+T  ],
\end{equation*}
and hence
\begin{equation}\label{eq. 2.43}
Y\in \bigg( [F_{p}^{c}- T, F_{p}^{c}+T  ]\cdot /(ab^{\top}+cd^{\top}) \bigg) \cap {\bf{Y}}=:{\bf{Y}}'\equiv \Gamma ({\bf{A}},{\bf{B}},{\bf{C}},{\bf{D}},{\bf{F}},{\bf{Y}}).
\end{equation}
So we obtain another enclosure ${\bf{Y}}'$ for $Y$. Since this works for all $Y$ with $A_{p}YB_{p}+C_{p}YD_{p}=F_{p}$, $A_{p}\in{\bf{A}}_{p}$, $B_{p}\in{\bf{B}}_{p}$, $C_{p}\in{\bf{C}}_{p}$, $D_{p}\in{\bf{D}}_{p}$ and $F_{p}\in{\bf{F}}_{p}$, thus we conclude that ${\bf{Y}}'$ will be another enclosure for the solution set of the interval system (\ref{eq. 2.39}).

If ${\bf{Y}}'$ defined by (\ref{eq. 2.43}) is strictly contained in ${\bf{Y}}$ then we may hope to get a further improved interval enclosure for the solution set of the interval system (\ref{eq. 2.39}) by repeating the above process, that leads to the following iteration
\begin{equation}\label{eq. 2.44}
\left\{ \begin{array}{ll}
{\bf{Y}}^{(0)}:= {\bf{Y}},\\
{\bf{Y}}^{(k+1)}:= \Gamma ({\bf{A}},{\bf{B}},{\bf{C}},{\bf{D}},{\bf{F}},{\bf{Y}}^{(k)}), \hspace{.7 cm}  k=0,1,\ldots .
\end{array} \right.
\end{equation}
Iteration (\ref{eq. 2.44}) is terminated whenever the distance between two successive iterations becomes smaller than a given tolerance.

\begin{theorem}\label{atheorem 2.6}
Consider the interval generalized Sylvester matrix equation (\ref{eq.1.5}) and let an initial enclosure ${\bf{Y}}\in{\mathbb{IK}}^{m\times n}$ for the solution set of the interval system (\ref{eq. 2.39}) be given. If the interval matrix ${\bf{Z}}\in{\mathbb{IK}}^{m\times n}$ is obtained by applying the iteration (\ref{eq. 2.44}) then $\Xi\subseteq U{\bf{Z}}V^{-1}$.
\end{theorem}

\begin{proof}
Due to the process leading to construction of the iteration (\ref{eq. 2.44}), we know that ${\bf{Z}}$ will be an enclosure for the solution set of the interval system (\ref{eq. 2.39}). This means that for every $A\in{\bf{A}}$, $B\in{\bf{B}}$, $C\in{\bf{C}}$, $D\in{\bf{D}}$ and $F\in{\bf{F}}$, the solution $\hat{Y}$ of the generalized Sylvester matrix equation $A_{p}YB_{p}+C_{p}YD_{p}=F_{p}$ satisfies $\hat{Y}\in{\bf{Z}}$, wherein $A_{p}=U^{-1}AU$, $B_{p}=V^{-1}BV$, $C_{p}=U^{-1}CU$, $D_{p}=V^{-1}DV$, $F_{p}=U^{-1}FV$ and $Y=U^{-1}XV$. This implies that the solution $\hat{X}$ of the generalized Sylvester matrix equation $AXB+CXD=F$ satisfies $\hat{X}\in U{\bf{Z}}V^{-1}$. Since $A$, $B$, $C$, $D$ and $F$ were chosen arbitrary, we conclude that $\Xi\subseteq U{\bf{Z}}V^{-1}$.
\end{proof}

It is worth noting that in another point of view, the proposed iterative approach can be applied for improving the (wide) enclosures that are obtained by the other methods for the interval system (\ref{eq. 2.39}) and so obtaining a tight enclosure for the interval generalized Sylvester matrix equation (\ref{eq.1.5}).

\begin{theorem}\label{atheorem 2.7}
The proposed iterative method in Theorem \ref{atheorem 2.6} requires $\mathcal{O}(m^{3}+n^{3})$ arithmetic operations per iteration.
\end{theorem}

\begin{proof}
The cost for the spectral decompositions (\ref{eq.2.14}) and (\ref{eq.2.15}) is cubic in the dimension of the involved matrices which adding together yields $\mathcal{O}(m^{3}+n^{3})$. Other operations between $m$-by-$m$, $n$-by-$n$ or $m$-by-$n$ matrices (multiplications, additions or point-wise divisions) cost $\mathcal{O}(m^{3}+n^{3})$. And since the costs of the other computational parts are negligible so the theorem is proved.
\end{proof}

\section{Numerical Tests}\label{section 3}
In this section, we present some numerical experiments to support the theoretical results and to illustrate the efficiency of the proposed methods in Section~\ref{section 2}.

In the following, we give some examples to compare the results obtained by the modified Krawczyk method proposed in Subsection \ref{subsection 2.2} (MKW), the iterative method proposed in Subsection \ref{subsection 2.4} (ITR) and \verb"VERMATREQN.m" code of \verb"VERSOFT" \cite{versoft} (VER) in terms of the executing time and quality of the obtained enclosures.\\

In Tables \ref{table 1}, \ref{table 2} and \ref{table 3} below, we compare the obtained enclosures by relative sums of radii with respect to the enclosure obtained by the modified Krawczyk method. That is, for an obtained enclosure matrix ${\bf{Y}}$ and the modified Krawczyk enclosure ${\bf{Z}}$, we display
\begin{equation*}
{\textrm{Ratio}}=\frac{\sum_{i,j}{\textrm{rad}}({\bf{Y}}_{ij})} {\sum_{i,j}{\textrm{rad}}({\bf{Z}}_{ij})}.
\end{equation*}
Also in the following figures, we display the log scale of the average radius meanR of the obtained enclosure ${\bf{X}}\in{\mathbb{IK}}^{m\times n}$ by each of the methods MKW, ITR and VER defined as
\begin{equation*}
{\textrm{meanR}}=\frac{\sum_{i,j}{\textrm{rad}}({\bf{X}}_{ij})}{mn}.
\end{equation*}
In the following tables all times are in seconds and the notation OM means that \verb"VERSOFT" failed because of ''out of memory''. Also \verb"VERSOFT" returns "NaN" when fails in a problem.

\begin{example}\label{example 3.1}
 Consider the interval Kalman-Yakubovich-conjugate matrix equation
\begin{equation}\label{eq.3.1}
{\bf{A}}X{\bf{B}}+X={\bf{C}},
\end{equation}
in which ${\bf{A}}$, ${\bf{B}}$ and ${\bf{C}}$ are obtained randomly by the following Matlab's functions
{\setlength\arraycolsep{2pt}
\begin{eqnarray*}
&& \verb"A1=4*rand(m,m)-3*ones(m,m);A2=A1+alpha*rand(m,m);A=infsup(A1,A2);"
\nonumber\\
&& \verb"B1=3*rand(m,m)-2*ones(m,m);B2=B1+alpha*rand(m,m);B=infsup(B1,B2);"\nonumber\\
&& \verb"C1=ones(m,m);C2=C1+alpha*rand(m,m);C=infsup(C1,C2);"
\end{eqnarray*}}
with \verb"alpha"=$10^{-6}$. Numerical results are reported in Table \ref{table 1} for various dimensions $m$. Using the methods MKW, ITR and VER for enclosing the solution set of (\ref{eq.3.1}), $\log$ scales of the average radius meanR (computational time) for each method is plotted against the dimension $m$ in the right side (left side) of Figure \ref{figure 1}.\\

\begin{table}[t]
\caption{Results for Example \ref{example 3.1}} \label{table 1}
\begin{center}
\begin{tabular}{llllllll}
\hline
$m$ &  &Times    & &  &  &Ratios   & \\
\cline{2-4}\cline{6-8}
 & VER & MKW & ITR  &  & MKW & ITR & VER \\
\hline
10 & 0.0227 &  0.0242 &0.0372   &  &1  &1.0000 &4.2031        \\
20 & 0.1813 &  0.0218 &0.0383   &  &1  &0.9999 &1.4935     \\
30 & 1.3430 &  0.0237 &0.0472   &  &1  &0.9999 &7.7188     \\
40 & 6.0125 &  0.0310 &0.0703   &  &1  &0.9999 &13.598     \\
50 &19.289  &  0.0390 &0.0867   &  &1  &0.9995 &0.9419     \\
60 & 55.364 &  0.0550 &0.1198   &  &1  &0.9998 &1.9177        \\
70 &111.83  &0.0724   &0.1173   &  &1  &0.9997 &11.965      \\
80 &292.53  &0.0762   &0.1429   &  &1  &0.9993 &1.9631     \\
90 &1117.0  &0.0969   &0.1858   &  &1  &0.9996 &NAN  \\
100&93265  & 0.1082   &0.1948   &  &1  &0.9994 &NAN  \\
120&OM      &0.1576   &0.2988   &  &1  &0.9994 &{\bf{--}}   \\
130&OM      &0.1958   &0.3271   &  &1  &1.0000 &{\bf{--}}   \\
140&OM      &0.2187   &0.3902   &  &1  &0.9992 &{\bf{--}}   \\
150&OM      &0.2958   &0.5080   &  &1  &0.9992 &{\bf{--}}   \\
160&OM      &0.3684   &0.6415   &  &1  &0.9993 &{\bf{--}}   \\
170&OM      &0.4321   &0.7387   &  &1  &1.0003 &{\bf{--}}   \\
180&OM      &0.4870   &0.8286   &  &1  &1.0044 &{\bf{--}}   \\
190&OM      &0.9364   &0.8746   &  &1  &1.0041 &{\bf{--}}   \\
200&OM      &0.5652   &1.0037   &  &1  &1.0147 &{\bf{--}}   \\
\hline
\end{tabular}
\end{center}
\end{table}

From the reported values in Table \ref{table 1}, we see that unless for very small dimensions the proposed methods in this paper are very much faster than \verb"VERSOFT" that confirms this fact that the  new methods have just a cost of $O(m^{3})$ arithmetic operations while \verb"VERSOFT" involves $O(m^{6})$ operations. Also from the displayed values for relative sums of radii, we find that MKW and ITR methods give tighter enclosures than those obtained by VER method for almost all dimensions $m$. And ITR method gives slightly narrower enclosures than the MKW method for almost dimensions $m$ but MKW method performs slightly faster than ITR method. Figure \ref{figure 1} shows that our approaches give smaller average radii than those by VER method, on the average, when apply for solving different problems.
\end{example}

\begin{example}\label{example 3.2}
Consider the interval Sylvester matrix equation
\begin{equation}\label{eq.3.2}
{\bf{A}}X+X{\bf{B}}={\bf{C}},
\end{equation}
in which ${\bf{A}}$, ${\bf{B}}$ and ${\bf{C}}$ are obtained similarly to Example \ref{example 3.1}. You can see the obtained results by executing the methods MKW, ITR and VER for enclosing the solution set of equation (\ref{eq.3.2}) in Table \ref{table 2}. Also Figure \ref{figure 2} shows the computational time and average radius meanR obtained by executing of each method versus dimensions $m$ based on $\log$ scale of $y$-axis, respectively in left side and right side.
\begin{figure}[t]
\includegraphics[height=6cm,width=13cm]{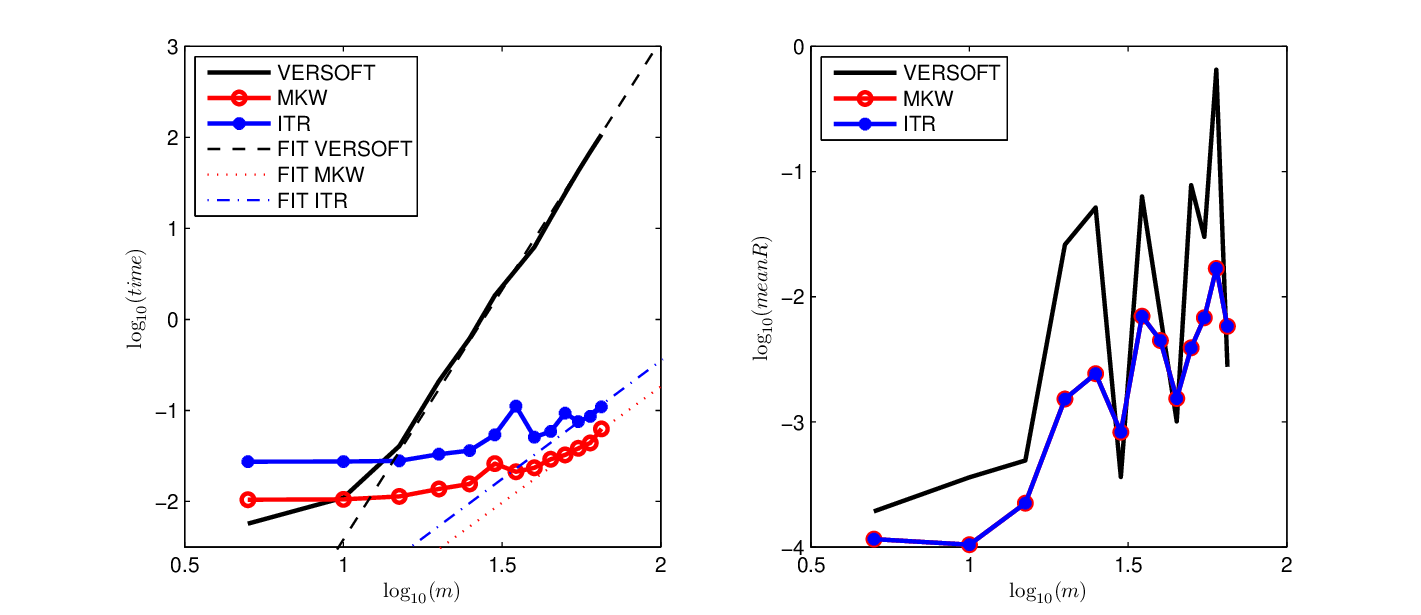}\\
\caption{$\log$ scale of the average radius (computing time) versus dimension $m$ at right (left) for Example \ref{example 3.1}.}
\label{figure 1}
\end{figure}
As one can see, again unless for very small dimensions the proposed methods in this paper are very faster than \verb"VERSOFT". The reported numbers in Table \ref{table 2} shows that VER method gives tighter enclosures than those obtained by MKW and ITR methods for almost dimensions as is shown also in Figure \ref{figure 2}. MKW method is faster than the ITR method whereas ITR method gives slightly narrower enclosures than MKW method. \verb"VERSOFT" fails from the 120th dimension onwards because of memory over.
\begin{table}[t]
\caption{Results for Example \ref{example 3.2}} \label{table 2}
\begin{center}
\begin{tabular}{llllllll}
\hline
$m$ &  &Times    & &  &  &Ratios   & \\
\cline{2-4}\cline{6-8}
 & VER & MKW & ITR  &  & MKW & ITR & VER \\
\hline
10 &  0.0191 &0.0141 &0.0374   &  &1  &1.0000 &  0.2088         \\
20 &  0.1715 &0.0172 &0.0403   &  &1  &1.0000 &  0.2023      \\
30 &  1.2069 &0.0200 &0.0462   &  &1  &0.9999 &  3.7486      \\
40 &  5.1551 &0.0296 &0.0623   &  &1  &0.9999 &  0.1956      \\
50 &  17.696 &0.0416 &0.0844   &  &1  &0.9992 &  20.052     \\
60 &  47.434 &0.0694 &0.1020   &  &1  &0.9996 &  0.6343         \\
70 &109.66   &0.1149 &0.1421   &  &1  &0.9997 &  0.0715      \\
80 &297.88   &0.0860 &0.1563   &  &1  &0.9996 &  0.2151     \\
90 &1.170.8  &0.1859 &0.1849   &  &1  &0.9996 &0.2134     \\
100&7920.3   &0.1191 &0.2139   &  &1  &0.9996 &0.1922     \\
110& 23048   &0.1924 &0.3297   &  &1  &0.9992 & 0.3361     \\
120&OM       &0.1602 &0.3329   &  &1  &0.9994 &{\bf{--}}   \\
130&OM       &0.2487 &0.3642   &  &1  &0.9992 &{\bf{--}}   \\
140&OM       &0.2422 &0.4145   &  &1  &0.9992 &{\bf{--}}   \\
150&OM       &0.2948 &0.5254   &  &1  &0.9992 &{\bf{--}}   \\
160&OM       &0.3520 &0.5642   &  &1  &0.9989 &{\bf{--}}   \\
170&OM       &0.4136 &0.6811   &  &1  &0.9990 &{\bf{--}}   \\
180&OM       &0.4858 &0.7605   &  &1  &0.9990 &{\bf{--}}   \\
190&OM       &0.5623 &0.9396   &  &1  &0.9993 &{\bf{--}}   \\
200&OM       &0.6099 &1.0008   &  &1  &1.0002 &{\bf{--}}   \\
\hline
\end{tabular}
\end{center}
\end{table}

\begin{figure}[t]
\includegraphics[height=6cm,width=13cm]{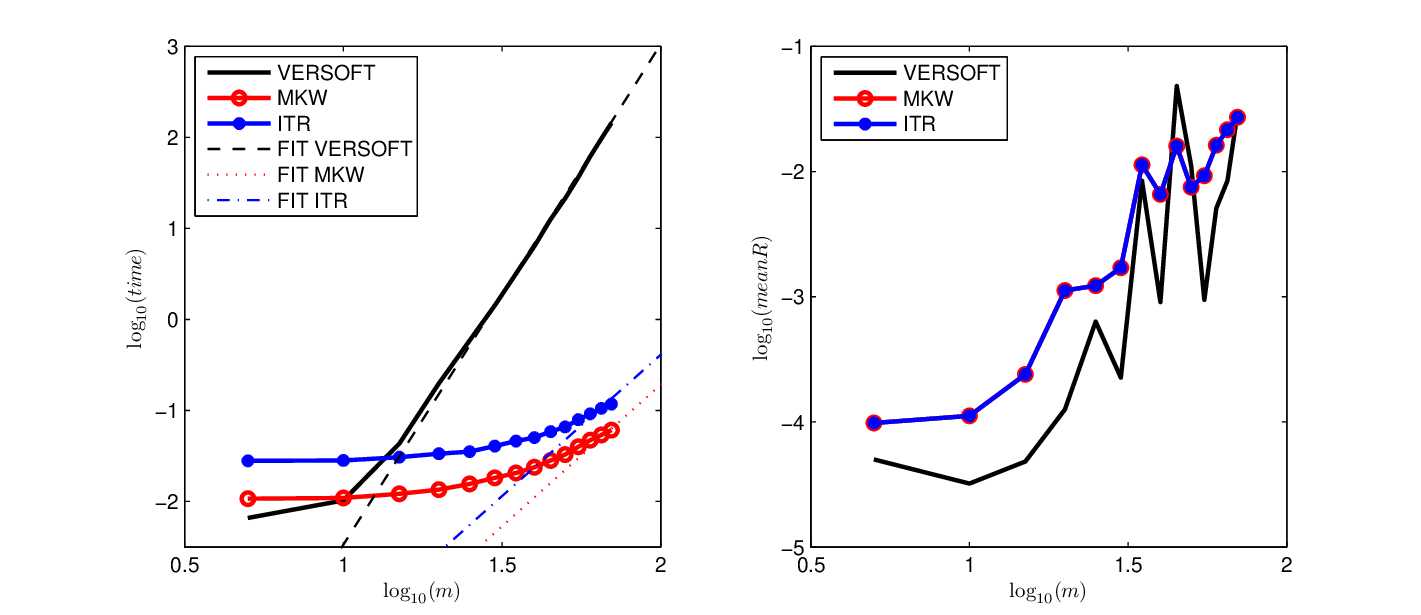}\\
\caption{$\log$ scale of the average radius (computing time) versus dimension $m$ at right (left) for Example \ref{example 3.2}.}
\label{figure 2}
\end{figure}
\end{example}

\begin{example}\label{example 3.3}
In this example, we consider the interval generalized Sylvester matrix equation
\begin{equation}\label{eq.3.3}
{\bf{A}}X{\bf{B}}+{\bf{C}}X{\bf{D}}={\bf{F}},
\end{equation}
in which ${\bf{A}}$, ${\bf{B}}$, ${\bf{C}}$, ${\bf{D}}$ and ${\bf{F}}$ are made by function \verb"gallery" of Matlab as follow
\begin{eqnarray*}
&& \verb"A1= gallery('parter',m)-ones(m,m);A2=A1+alpha* gallery('lehmer',m);"
\nonumber\\
&&\verb"A=infsup(A1,A2);C=A+midrad(0,alpha)"
\nonumber\\
&& \verb"B1= gallery('parter',m)-ones(m,m);B2=B1+alpha* gallery('lehmer',m);"
\nonumber\\
&&\verb"B=infsup(B1,B2);D=B+midrad(0,alpha)"
\nonumber\\
&& \verb"F1=gallery('lehmer',m);F2=F1+alpha*F1;F=infsup(F1,F2);"
\end{eqnarray*}
with \verb"alpha"=$10^{-6}$. The obtained results by executing the methods MKW, ITR and VER for finding outer estimation for the solution set of equation (\ref{eq.3.3}) are shown in Table \ref{table 3} for various dimensions $m$.

From the presented results for execution times in Table \ref{table 3}, we see that the proposed methods in this paper are substantially faster than \verb"VERSOFT". And the reported numbers for relative sums of radii show that \verb"VERSOFT" gives tighter enclosures than MKW and ITR methods. ITR method yields tighter enclosures than the MKW method, on the average. \verb"VERSOFT" fails from the 120th dimension onwards because of memory over.
\begin{table}[t]
\caption{Results for Example \ref{example 3.3}} \label{table 3}
\begin{center}
\begin{tabular}{llllllll}
\hline
$m$ &  &Times    & &  &  &Ratios   & \\
\cline{2-4}\cline{6-8}
 & VER & MKW & ITR  &  & MKW & ITR & VER \\
\hline
10 &  0.0768 &0.0169 &0.0356   &  &1  &1.0000 & 0.1370           \\
20 &  0.1924 &0.0178 &0.0397   &  &1  &1.0000 & 0.0904        \\
30 &  1.1841 &0.0254 &0.0479   &  &1  &0.9999 & 0.0702        \\
40 &  5.0135 &0.0341 &0.0655   &  &1  &0.9999 & 0.0578        \\
50 & 16.405  &0.0760 &0.0816   &  &1  &0.9998 & 0.0491       \\
60 & 47.203  &0.0734 &0.1040   &  &1  &0.9998 & 0.0427           \\
70 &109.30   &0.1610 &0.1396   &  &1  &0.9997 & 0.0378        \\
80 &282.76   &0.2272 &0.1547   &  &1  &0.9996 & 0.0339       \\
90 &  1459.7 &0.1134 &0.1914   &  &1  &0.9995 &  0.0307    \\
100& 7068.3  &0.3256 &0.2301   &  &1  &0.9994 &  0.0281    \\
120&OM       &0.2634 &0.3146   &  &1  &0.9993 &{\bf{--}}   \\
130&OM       &0.2327 &0.3754   &  &1  &0.9994 &{\bf{--}}   \\
140&OM       &0.2920 &0.4466   &  &1  &0.9998 &{\bf{--}}   \\
150&OM       &0.4523 &0.5480   &  &1  &1.0007 &{\bf{--}}   \\
160&OM       &0.4221 &0.6826   &  &1  &0.9988 &{\bf{--}}   \\
170&OM       &0.5117 &0.7873   &  &1  &0.9996 &{\bf{--}}   \\
180&OM       &0.5854 &0.8809   &  &1  &1.0013 &{\bf{--}}   \\
190&OM       &0.6712 &0.9977   &  &1  &1.0048 &{\bf{--}}   \\
200&OM       &0.8610 &1.1579   &  &1  &1.0112 &{\bf{--}}   \\
\hline
\end{tabular}
\end{center}
\end{table}

\end{example}

\section{Conclusion}\label{section 4}
This paper was addressed to the united solution set to the interval generalized Sylvester matrix equation (\ref{eq.1.5}). We gave necessary conditions characterizing the solution set, and we also gave a sufficient condition under which this solution set is bounded. We proposed a modified Krawczyk operator on the preconditioned system to compute outer estimations for the solution set such that keeps the computational complexity down to cubic. We then presented an iterative method on the same preconditioned system for enclosing the solution set. The proposed methods can be applied for many other interval systems that are special cases of (\ref{eq.1.5}). Numerical experiments show the effectiveness of the new approaches on the execution times and also on quality of the computed enclosures.

\subsubsection*{Acknowledgments.} 

Marzieh Dehghani-Madiseh would like to thank her supervisor prof.\ Mehdi Dehghani for his support and encouragement during this project.

\end{document}